\documentclass[11pt]{amsart}
\usepackage{stmaryrd}
\usepackage{amssymb}
\usepackage{csvsimple}
\usepackage{siunitx}
\usepackage{hyperref}
\usepackage{url}
\usepackage{graphicx,color}
\usepackage{lineno}           
\usepackage[style=trad-abbrv,url=false,doi=true,
  eprint=true,isbn=false,backend=biber]{biblatex}

\graphicspath{{pics/},{figs/}}
\def\R{\mathbb{R}}

\def\N{\mathbb{N}}

\def\Rinf{\R\cup \{+\infty\}}

\def\cI{\mathcal{I}}
\def\cJ{\mathcal{J}}

\def\cL{\mathcal{L}}

\def\cN{\mathcal{N}}

\def\cS{\mathcal{S}}
\def\cT{\mathcal{T}}

\def\a{\alpha}
\def\b{\beta}
\def\g{\gamma}
\def\d{\delta}

\def\k{\kappa}

\def\s{\sigma}
\def\p{\partial}
\def\o{\omega}
\def\veps{\varepsilon}

\def\vphi{\varphi}

\def\G{\Gamma}

\def\wto{\rightharpoonup}
\def\transp{{\sf T}}

\newcommand{\dv}[1]{\,{\mathrm d}#1}

\newcommand{\wcheck}[1]{#1\hspace{-.8ex}\mbox{\huge {\lower.45ex \hbox{$\textstyle \check{}$}}} \hspace{.5ex}}

\newcommand{\jump}[1]{\llbracket#1\rrbracket}   

\DeclareMathOperator{\diver}{div}

\DeclareMathOperator{\trace}{tr}

\let\oldmarginpar\marginpar
\renewcommand\marginpar[1]{
  \oldmarginpar[\raggedleft\footnotesize #1]
  {\raggedright\footnotesize #1}}

\newtheorem{definition}{Definition}
\newtheorem{lemma}[definition]{Lemma}
\newtheorem{proposition}[definition]{Proposition}
\newtheorem{theorem}[definition]{Theorem}

\newtheorem{remark}[definition]{Remark}
\newtheorem{remarks}[definition]{Remarks}
\newtheorem{example}[definition]{Example}

\numberwithin{definition}{section}
\definecolor{modmag}{RGB}{179,0,229}

\renewcommand{\text}{\textnormal}

\def\tI{\widetilde{I}}

\def\oh{\overline{h}}
\def\arg{{\rm arg}}

\def\dkt{{\rm dkt}}
\def\dg{{\rm dg}}
\def\tcI{\widetilde{\cI}}
\DeclareMathOperator{\Diver}{Div}
\def\oh{{\o_h}}
\def\stab{{\rm stab}}
\def\inv{{\rm inv}}
\def\trace{{\rm tr}}
\def\tV{\widetilde{V}}
\def\av{{\rm av}}
\def\po{{\rm p1}}

\def\oo{\overline{\o}}
\def\tchi{\widetilde{\chi}}

\newcommand{\jj}[1]{\jump{#1}}
\newcommand{\aver}[1]{\{#1\}} 

\def\nh{\nabla_h}

\def\pnnh{\partial_n \nabla_h}
\def\pnDh{\partial_n \Delta_h}
\def\cSh{{\cup \cS_h}}

\begin{document}
\title[Avoiding the plate paradox]{Necessary and sufficient conditions for avoiding 
Babuska's paradox on simplicial meshes}
\author{S\"oren Bartels}
\address{Abteilung f\"ur Angewandte Mathematik,
Albert-Ludwigs-Universit\"at Freiburg, Hermann-Herder-Str.~10,
79104 Freiburg i.~Br., Germany}
\email{bartels@mathematik.uni-freiburg.de}
\author{Philipp Tscherner}
\address{Abteilung f\"ur Angewandte Mathematik,
Albert-Ludwigs-Universit\"at Freiburg, Hermann-Herder-Str.~10,
79104 Freiburg i.~Br., Germany}
\email{philipp.tscherner@mathematik.uni-freiburg.de}
\date{\today}
\renewcommand{\subjclassname}{
\textup{2010} Mathematics Subject Classification}
\subjclass[2010]{65N12, 65N30, 35J40, 74K20}
\begin{abstract}
It is shown that discretizations based on variational or weak formulations
of the plate bending problem with simple support boundary conditions do 
not lead to failure of convergence when polygonal domain approximations are
used and the imposed boundary conditions are compatible with the nodal interpolation 
of the restriction of certain regular functions to approximating domains. 
It is further shown that this is optimal in the
sense that a full realization of the boundary conditions leads to failure
of convergence for conforming methods. The abstract conditions imply that 
standard nonconforming and discontinuous Galerkin methods converge correctly
while conforming methods require a suitable relaxation of the boundary condition.
The results are confirmed by numerical experiments. 
\end{abstract}
\keywords{Plate bending, finite elements, domain approximation, Babuska paradox}

\maketitle

\section{Introduction} 
The plate or Babu\v{s}ka paradox refers to the failure
of convergence when a linear bending problem with simple support boundary 
conditions on a domain with curved boundary is approximated using
a sequence of problems on approximating polygonal domains, cf.~\cite{Babu63}.
An explanation 
for this is an insufficient consistency in the approximation of the 
curvature of the boundary. Remarkably, numerical experiments show that
typical nonconforming and discontinuous Galerkin methods converge
correctly on sequences of simplicial meshes~\cite{Wiss23-msc}. It is the goal 
of this article to identify criteria for numerical methods that avoid the 
occurrence of the paradox and explain the observed convergence. The main
result is that a suitable relaxation in the treatment of the boundary
conditions avoids the failure of convergence independently of regularity
properties. While this is naturally 
satisfied by canonical realizations of
nonconforming and discontinuous Galerkin methods, in the
case of a conforming method, the boundary condition need to be restricted
to the boundary vertices only, which is typically consistent with the
nodal interpolation of certain regular functions. 

Small elastic deflections $u:\o \to \R$ of a thin plate are described by
a minimization of the energy functional 
\[
I(v) = \frac{\s}{2} \int_\o |\Delta v|^2 \dv{x} + \frac{1-\s}{2} \int_\o |D^2 v|^2 \dv{x}
- \int_\o f v \dv{x}
\]
in a set $V \subset H^2(\o)$ whose definition involves appropriate boundary
conditions. The material parameter $\s$ is the Poisson ratio of the elastic
material and we assume for simplicity that $0\le \s <1$. 
So called conditions of simple support prescribe the 
deflection on the boundary, in the simplest setting via imposing
\[
u|_{\p \o} = 0,
\]
so that $V = H^2(\o)\cap H^1_0(\o)$. Clamped boundary conditions additionally
prescribe the normal of the deformed plate along the boundary, e.g., 
via $\nabla u = 0$ on $\p\o$, so that in this case we have $V = H^2_0(\o)$. 
Because of the density of compactly supported smooth functions in $H^2_0(\o)$
it is straightforward to show that domain approximations are not critical
for clamped boundary conditions. 

In the case of simple support boundary conditions, the unique minimizer $u\in V$
is characterized by the Euler--Lagrange equations 
\[
\Delta^2 u = f \ \text{in }\o, \quad u = \Delta u - (1-\s) \k \p_n u = 0 \ \text{on }\p\o,
\]
where $\k$ denotes the curvature of the boundary (positive for locally convex boundary 
parts), and $\p_n u = \nabla u \cdot n$ is
the outer normal derivative along the boundary. For approximating polygonal domains $\o_m$
one has $\k =0$ away from corner points in which discrete curvature
is captured by Dirac contributions. Owing to the boundary condition $u=0$ on $\p\o_m$
we have however $\nabla u=0$ and hence $\p_n u = 0$ in those points. Thus, formally
the term involving $\k$  in the natural boundary condition disappears for polygonal 
domain approximations. In convex domains $\o_m$ the 
solutions $u_m$ can then be efficiently computed using an operator
splitting, i.e., by introducing the variables $w_m = -\Delta u_m$ and successively 
solving the problems 
\[\begin{split}
(a) & \quad -\Delta w_m = f \ \text{ in }\o_m, \quad  w_m = 0 \ \text{on }\p\o_m, \\
(b) & \quad -\Delta u_m = w_m \ \text{in }\o_m, \quad u_m = 0 \ \text{on }\p\o_m.
\end{split}\]
Since the approximation of boundaries is not critical for Poisson problems with
Dirichlet boundary conditions,
it follows that limits $(u_\infty,w_\infty)$ of the sequence $(u_m,w_m)$ provide the
solution of the problem
\[
\Delta^2 u_\infty = f \ \text{in }\o, \quad u_\infty= \Delta u_\infty = 0 \ \text{on }\p\o.
\]
In particular, $u_\infty$ is independent of $\s$ and in general different from
the true solution~$u$. A rotationally symmetric example with $f=1$ in 
the unit disk has been determined in~\cite{BabPit90}, 
cf. Figure~\ref{fig:sol_and_wrong} and Section~\ref{sec:num_ex}.
This discrepancy is referred to as the plate
or Babu\v{s}ka paradox. A comprehensive discussion of the analytical aspects of this
and related phenomena can be found in~\cite{NaSwSt11,GaGrSw10-book}.
For domain approximations using piecewise quadratic 
boundary curves the piecewise curvature correctly approximates $\k$ in the
sense of functions and coefficients of Dirac contributions associated with corner
points decay superlinearly and are therefore irrelevant. We refer the reader
to~\cite{Scot77,BrNeSu13,ArnWal20} for related ideas and numerical methods. 
Other methods to avoid the paradox introduce certain Dirac contributions
in boundary condition obtained via asymptotic expansions, cf.~\cite{MazNaz86}, 
are based on nonconforming methods~\cite{DavPit00,Davi02,ArnWal20},
or impose the boundary condition in a modified or penalized 
form~\cite{StrFix73,Rann79,UtkCar83}.
The framework of variational convergence adopted here avoids conditional results
based on regularity properties, does not require extension operators,
and leads to weak conditions on penalty parameters. 

\begin{figure}[htb]
\includegraphics[width=4cm]{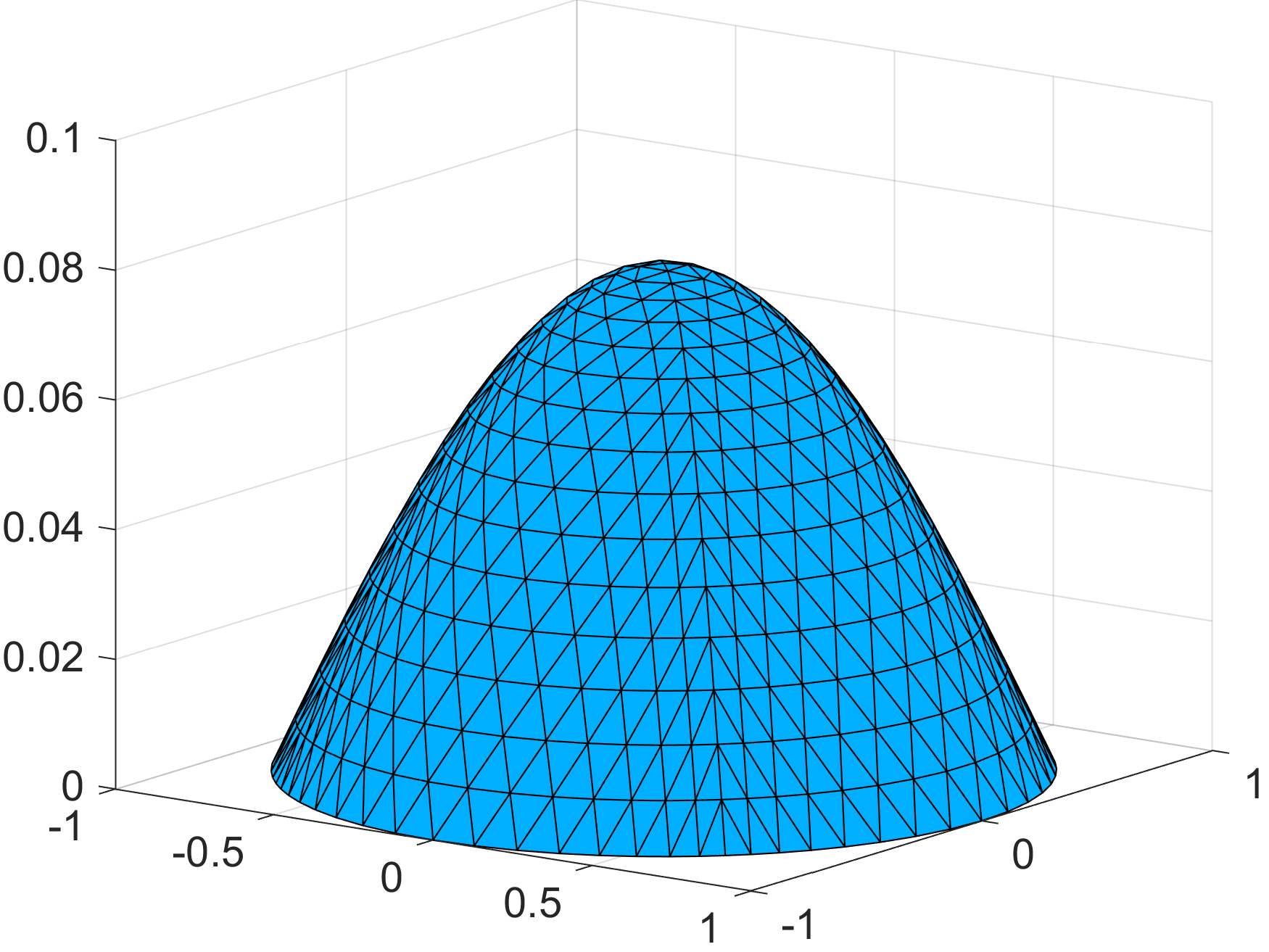} \hspace{.4cm}
\includegraphics[width=4cm]{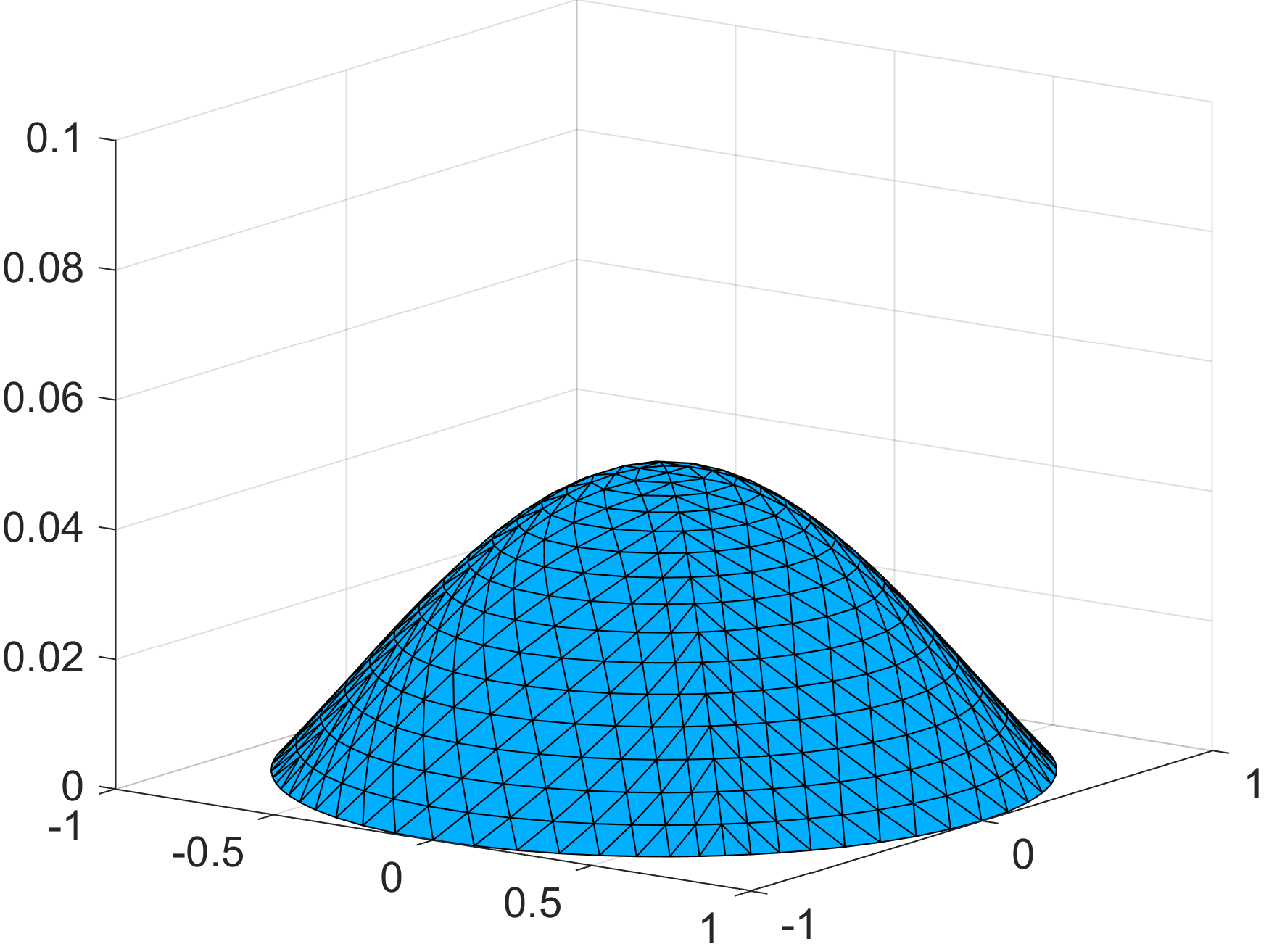}
\caption{\label{fig:sol_and_wrong} Interpolants of the solution of a plate bending problem
with simple support boundary conditions (left) and of the incorrect solution obtained as a
limit of problems on polygonal domain approximations (right).}
\end{figure}

The failure of convergence also occurs when working with the variational problems
or weak formulations resulting from replacing $\o$ by $\o_m$ in 
the energy functional $I$, i.e., considering the minimization of 
\[
I_m(v) = \frac{\s}{2} \int_{\o_m} |\Delta v|^2 \dv{x} 
+ \frac{1-\s}{2} \int_{\o_m} |D^2 v|^2 \dv{x} - \int_{\o_m} f_m v \dv{x}
\]
in the set of functions $v_m \in V_m =  H^1_0(\o_m) \cap H^2(\o_m)$. An integration
by parts and the density of $H^3$ regular functions in $V_m$ reveal 
boundary contributions that vanish for the functionals $I_m$ but not for the
original funtional~$I$, cf.~\cite{Davi02,CoNiSw19}. Hence, variational convergence 
cannot hold. This argument
also applies to finite element methods that are based on subspaces of $V_m$. 
It is thus a necessary condition that approximations are nonconforming.

To avoid the failure of convergence it suffices to introduce a nonconformity in 
the treatment of the boundary condition by reducing it 
to corner points or more generally to $\p\o_m\cap \p\o$,   
i.e., using the admissible space
\[
\tV_m  = \big\{v\in H^2(\o_m): v = 0 \text{ on }\p\o_m \cap \p\o \big\}.
\]
To show that the minimization of $I_m$ in $\tV_m$ converges to the minimization
of $I$ on $V = H^2(\o) \cap H^1_0(\o) $ we use the concept of $\G$-convergence which
avoids imposing regularity conditions.
We assume that $\o_m\subset \o$ such that corner points of $\p\o_m$ belong to 
$\p\o$, i.e., that $\o$ is convex, and always extend functions and derivatives
trivially by zero to $\o$.
The first step consists in showing that accumulation points $v \in L^2(\o)$ of 
sequences $(v_m)$ satisfy the lim-inf inequality
\[
I(v) \le \liminf_{m\to \infty} I_m(v_m).
\]
This is in fact straightforward since $\|v_m\|_{H^1(\o_m)} \le c_{21} \|D^2 v_m\|_{L^2(\o_m)}$
and $D^2 v_m \wto D^2 v$ after selection of a subsequence. By carrying out a
nodal interpolation $\cI_hv \in H^1_0(\o_m)$ on triangulations $\cT_m$ of $\o_m$ it 
follows that the boundary conditions are correctly satisfied in the limit, i.e., $v|_{\p\o}=0$
so that $v\in V$. 
The second step requires showing that the lower bound is attained for every
$v\in V$. For this, it suffices to realize that owing to the definition of $\tV_m$ 
the restrictions $v_m = v|_{\o_m}$ belong to $\tV_m$, and satisfy
$D^2 v_m \to D^2 v$ in $L^2(\o)$ since $\|D^2 v_m\|_{L^2(\o\setminus \o_m)} \to 0$ 
as well as
\[
I(v) = \lim_{m\to \infty} I_m(v_m).
\]
This establishes $\G$ convergence of $I_m$ to $I$ with respect to 
strong convergence in $L^2(\o)$. An immediate consequence is that minimizers
for $I_m$ converge to the minimizer of $I$. It is straightforward to check 
that the convergence is strong in $H^1$ on compact subsets of $\o$.

The remarkable difference between finite element approximations with boundary
conditions imposed on the entire boundary and the corner points of a pentagon
is illustrated in Figure~\ref{fig:arg_nodal_vs_full}. While the first one
is close to the restriction of the exact solution to the pentagon, the second
one misses the exact maximal value by a factor ten. 

\begin{figure}[htb]
\includegraphics[width=4cm]{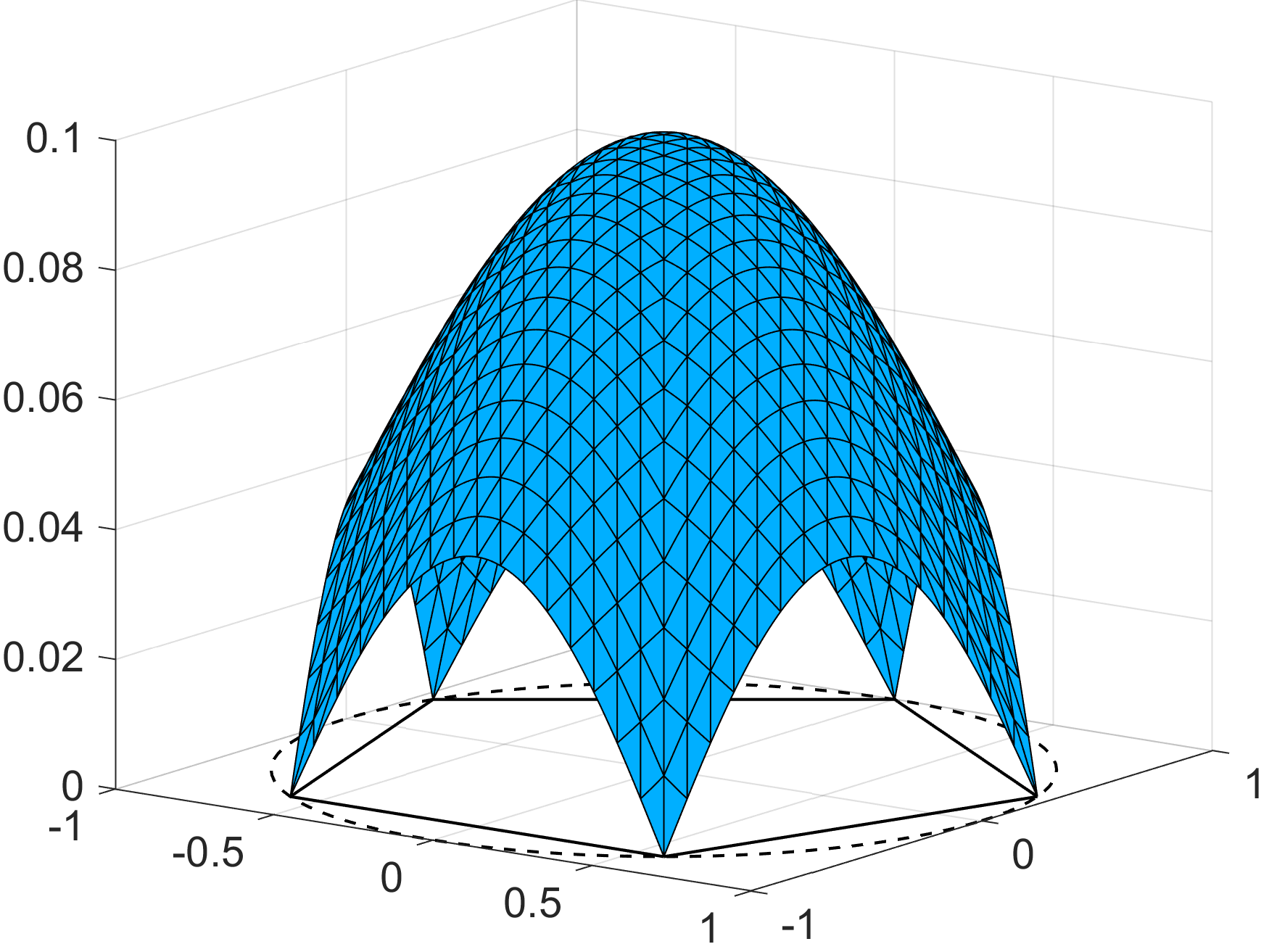} \hspace{.4cm}
\includegraphics[width=4cm]{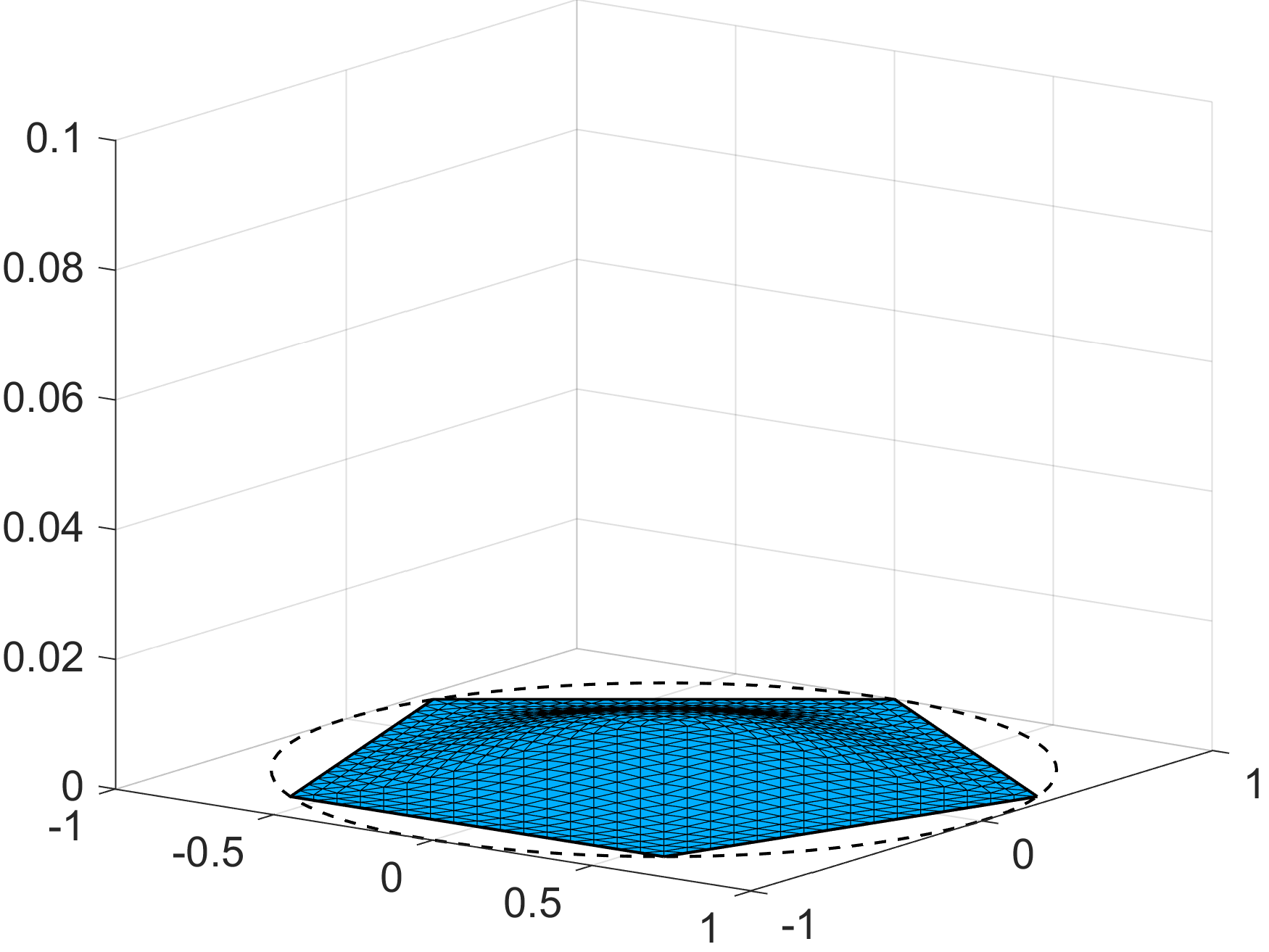}
\caption{\label{fig:arg_nodal_vs_full} Finite element solutions on a pentagon
imposing the boundary conditions at the corner points (left) and along the entire 
boundary (right).}
\end{figure}

Under mild additional integrability conditions on $f$ and $D^2 u$, error estimates
can be derived for the difference of $u_m$ and $u|_{\o_m}$. Letting $a_m$ and
$a$ denote the bilinear forms associated with the energy functionals $I_m$ and
$I$, noting that $u|_{\o_m}$ is admissible in the approximating problem,
and assuming that $E_m: H^2(\o_m)\to H^2(\o)\cap H^1_0(\o)$ are uniformly 
bounded extension operators, straightforward calculations lead to, e.g., if
$\s=0$, 
\[\begin{split}
\|D^2 (u - & u_m)  \|_{L^2(\o_m)}^2   = a_m(u-u_m,u-u_m)  \\
&= \int_{\o\setminus \o_m} f (u-E_m u_m) \dv{x} -  \int_{\o\setminus \o_m} D^2 u :D^2 (u- E_mu_m)\dv{x} \\
&\le c_E \, |\o\setminus \o_m|^{1/2} \big( \| f\|_{L^\infty(\o)} + \|D^2 u\|_{L^\infty(\o)}\big)  \|D^2 (u-E_m u_m)\|.
\end{split}\]
For simplicial triangulations $\cT_{h_m}$ with maximal mesh-size $h_m>0$ 
that define the subdomains $\o_m$ of the domain $\o$ with piecewise $C^2$ boundary,
we have the area difference estimate $|\o\setminus \o_m| \le c h^2$,
so that the domain approximation leads to an error contribution of at least
linear order. We refer the reader to~\cite{Rann79} for related estimates based
on the use of a Strang lemma. 

An important aspect in the transfer of convergence proofs to finite element settings
is the construction of a recovery
sequence via the interpolation of the restriction of functions in $H^3(\o)\cap V$
to $\o_m$. In this restriction only the zero boundary values at the corners
are captured and only these information are seen in the interpolation process.
The discrete admissible set thus has to be appropriately defined to ensure that 
the interpolants belong to it. The convergence proof given above serves as
a template to derive an abstract convergence theory that can be applied to 
various finite element methods. The sufficient conditions are that (1)~the approximating
problems are uniformly coercive in $H^1_0(\o)$, (2)~possible discretizations of
second order derivatives are stable, and (3)~that interpolation operators
map functions from $H^3(\o)\cap V$ into the discrete admissible sets.

As an alternative to imposing the boundary condition in the corner points, one may
impose it via penalty terms with suitably chosen penalty parameter. Letting 
\[
I_{m,\veps}(v) = I_m(v) + \frac{1}{2\veps} \int_{\p\o_m} v^2 \dv{s}
\]
one establishes a $\G$-convergence result if there exists a family of bounded
linear operators $F_m: H^1(\o_m) \to H^1(\o)$ that map traces $v|_{\p\o_m}$ boundedly
to traces $F_m(v)|_{\p\o}$. The condition implies that limits of sequences of functions
with uniformly bounded energies have vanishing traces on $\p\o$. The restrictions $v_m = v|_{\o_m}$
define a recovery sequence for a function $v\in H^3(\o)\cap H^1_0(\o)$ if 
$\veps$ is chosen such that $h_m^4 \veps \to 0$ as $h_m\to 0$ 
since $\|v_m\|_{L^\infty(\p\o_m)} \le h_m^2$. If quadrature is used in the penalty term
it realizes a penalized variant of imposing the boundary condition in the corner
points and in fact no restrictions on the parameters are required. 
These observations explain why discontinuous Galerkin 
methods, that impose the boundary conditions via penalty terms, converge 
correctly on sequences of polygonal domain approximations. 

The outline of this article is as follows. Some auxiliary results are collected
in Section~\ref{sec:aux}. The abstract convergence theory and a result about
failure of convergence are stated in Section~\ref{sec:abstr_res}. The application
of the framework to conforming, nonconforming, and discontinuous Galerkin methods
is discussed in Section~\ref{sec:methods}. Numerical experiments that confirm the 
theoretical results are reported in Section~\ref{sec:num_ex}. 

\section{Auxiliary results}\label{sec:aux}
We cite in this section an important density result from~\cite{CoNiSw19} together
with a boundary representation formula and collect some basic operators and
estimates related to finite element methods. Throughout, we use standard notation
for Sobolev and Lebesgue spaces. We let $(\cdot,\cdot)_A$ denote the $L^2$ inner
product on a set $A$ and occasionally abbreviate 
\[
\|v\| = \|v\|_{L^2(\o)}, \quad (v,w) = (v,w)_\o
\]
for functions $v,w \in L^2(\o)$ possibly obtained as trivial extensions of
functions defined in subsets $\o'\subset \o$. We always assume that $\o$ is a
convex and bounded Lipschitz domain with piecewise $C^{2,1}$ boundary and
finitely many corner points. We note that we have the Poincar\'e inequality
\begin{equation}\label{eq:poincare_h2}
\|\nabla v \| \le c_{\rm P} \|\Delta v\|
\end{equation}
for all $v\in H^2(\o)\cap H^1_0(\o)$ with a constant $c_{\rm P}>0$ that is
uniformly bounded for families of convex domains whose inner and outer
diameters are uniformly bounded. We occasionally use the symbol $\lesssim$ to express
an inequality that holds up to a generic constant factor. 

\subsection{Bending energy}
Crucial for our analysis is a density result for certain regular functions in 
the set $H^2(\o)\cap H^1_0(\o)$, cf.~\cite[Thm. B.5]{CoNiSw19}.

\begin{theorem}[Density] \label{thm:density}
The set $H^3(\o)\cap H^1_0(\o)$ is dense in $H^2(\o)\cap H^1_0(\o)$.
\end{theorem}

An important consequence of this result is the following formula from~\cite{Davi03}
and~\cite[Lemma 3.8]{CoNiSw19},
that provides a representation of the total curvature of the graph of a function
by a boundary integral.  

\begin{lemma}[Boundary representation]\label{la:bdy_repr}
For $v\in H^2(\o)\cap H^1_0(\o)$ we have that
\[
- \int_\o  \det D^2 v \dv{x}
= \frac12 \int_{\p\o} \k (\p_n v)^2 \dv{s}.
\]
\end{lemma}

\begin{proof}[Proof (sketched)]
Integration by parts shows for $v\in H^3(\o)\cap H^1_0(\o)$ that
\[
- \int_\o \det D^2 v \dv{x}  = \frac12 \int_\o \diver (JD^2 v J \nabla v) \dv{x} 
= \frac12 \int_{\p\o} (J D^2v J \nabla v) \cdot n \dv{s},
\]
where $J$ denotes the clockwise rotation by $\pi/2$. If $p:(\a,\b)\to \R^2$ is an
arclength parametrization of a boundary segment, we have that
\[
0 = \frac{d^2}{ds^2} (v\circ p) = (p')^\transp (D^2 v \circ p)  p' + (\nabla v \circ p) \cdot p''.
\]
Using $p' = J n$, $p''= \k n$, and $J \nabla v = (\p_n v) J n$ yields the formula.  
\end{proof}

The elementary relation $|D^2 v|^2 = (\Delta v)^2 - 2 \det D^2v$ and 
the lemma imply that the functionals $I$ defined on $\o$ and $I_m$ defined on
polygonal domains $\o_m$ can be represented via 
\begin{equation}\label{eq:fl_curv_term}
I(v) = \frac{1}{2} \int_\o |\Delta v|^2 \dv{x} + \frac{1-\s}{2} \int_{\p\o} \k (\p_n v)^2 \dv{x}
\end{equation}
for $v\in H^2(\o)\cap H^1_0(\o)$ and 
\begin{equation}\label{eq:fl_no_curv_term}
I_m(v) = \frac{1}{2} \int_{\o_m} |\Delta v|^2 \dv{x} 
\end{equation}
for $v\in H^2(\o_m)\cap H^1_0(\o_m)$ and polygonal domains $\o_m$. To
establish the failure of convergence of the functionals $I_m$ and the
correct convergence of the modified functionals $\tI_m$, 
we use the framework of $\G$-convergence, cf.~\cite{AtBuMi06}.

\begin{definition}[$\Gamma$-convergence]
The sequence of functionals $I_m:L^2(\o) \to \Rinf$ is {\em $\G$-convergent}
to $I:L^2(\o)\to \Rinf$ with respect to strong convergence in $L^2(\o)$ 
if the following conditions hold: \\
(i) If $v_m \to v$ in $L^2(\o)$ then we have $I(v)\le \liminf_{m\to \infty} I_m(v_m)$. \\
(ii) For every $v\in L^2(\o)$ there exists a sequence $(v_m)_{m\ge 0}$ with $v_m\to v$ and
$I(v)= \lim_{m\to \infty} I_m(v_m)$. 
\end{definition}

If the sequence of functionals $I_m$ is uniformly coercive, then it follows directly
that (almost) minimizers for $I_m$ converge to minimizers for $I$. 

\subsection{Finite element spaces}
The finite element spaces considered below are defined on regular triangulations 
$\cT_h$ consisting of triangles whose unions define bounded polygonal domains
$\o_h$. The index $h>0$ indicates a maximal mesh size that is assumed to converge
to zero in a sequence $(\cT_h)_{h>0}$. 
We further assume that the triangulations are uniformly shape regular, i.e., 
that the ratios of diameters and inner radii are uniformly bounded. 
Typical finite element spaces are subspaces of spline spaces
\[
\cS^{\ell,k}(\cT_h) = \big\{ v_h \in C^k(\oo_h): v_h|_T \in P_\ell(T)\big\},
\]
where $P_\ell(T)$ denotes the space of polynomials of total degree at most $\ell$ 
on $T$. If $k=0$ then this superscript is omitted. 
An important space is the $P1$ finite element space with vanishing traces
\[
\cS^1_0(\cT_h) = \cS^1(\cT_h) \cap H^1_0(\o_h).
\]
We let $\cN_h$ denote the set of vertices in $\cT_h$ and note that a function
$v_h \in \cS^1(\cT_h)$ satisfies $v_h|_{\p\o_h}= 0$ if and only if $v_h(z)=0$
for all $z\in \cN_h\cap \p \o_h$. The set of sides of elements in $\cT_h$ is denoted by
$\cS_h$. We impose the following canonical conditions that 
relate the triangulations $\cT_h$ to the domain $\o$:
\begin{itemize}
\item[(T1)] The boundary vertices of the triangulation $\cT_h$ belong to the boundary
of $\o$, i.e., $\cN_h \cap \p\o_h \subset \p\o$.
\item[(T2)] The corner points $c_0,c_1,\cdots,c_\ell$ of $\o$ belong to the set of
vertices, i.e., $c_0,c_1,\dots,c_\ell \subset \cN_h$. 
\end{itemize}
With the nodal basis $(\vphi_z)_{z\in \cN_h}$
of $\cS^1(\cT_h)$ the nodal interpolation operator 
$\cI_h^\po:C(\overline{\o}_h) \to \cS^1(\cT_h)$ is defined via
\[
\cI_h^\po v = \sum_{z\in \cN_h} v(z) \vphi_z.
\]
Note that if $v\in H^3(\o)\cap H^1_0(\o)$ then we have that
\[
\cI_h^\po(v|_{\o_h}) \in \cS^1_0(\cT_h).
\]
For all $v\in H^2(\o_h)$ we have the interpolation estimate
\[
\|v-\cI_h^\po v \|_{H^1(\o_h)} \le c_{\cI} h \|D^2 v\|_{L^2(\o_h)}
\]
with a constant $c_{\cI}$ that
remains bounded as $h\to 0$. A node averaging operator 
$\cJ_h^\av:\cL^k(\cT_h) \to \cS^1_0(\cT_h)$ defined on the space of 
elementwise polynomial functions of degree at most $k$ is defined via the nodal 
values $v_z=0$ for $z\in \cN_h\cap \p\o_h$ and
\[
v_z = \frac{1}{n_z} \sum_{T: z\in T} v_h|_T(z)
\]
for inner nodes $z \in \cN_h\setminus \p\o_h$ and the number $n_z$ of elements containing $z$.
The jumps of $v_h$ are for sides $S\in \cS_h$ defined by 
\[
\jump{v_h}(x) =  
\begin{cases}
v_h(x)  & \text{if } S\subset \p\o_h, \\
\lim_{\veps \to 0} v_h(x+\veps n_S) - v_h(x-\veps n_S) & \text{if } S \not \subset \p\o_h,
\end{cases}
\]
with fixed unit normal vectors $n_S$, $S\in \cS_h$. We then have that
\[
\|v_h - \cJ_h^\av v_h \|_{L^2(\o_h)}^2 \le c_k h^2  \sum_{S\in \cS_h} h_S^{-1} \|\jump{v_h}\|_{L^2(S)}^2,
\]
cf., e.g.,~\cite{ErnGue04}.
Averages of discontinuous functions $v_h$ are on sides $S\in \cS_h$ defined via
\[
\aver{v_h}(x) =  
\begin{cases}
v_h(x)  & \text{if } S\subset \p\o_h, \\
\lim_{\veps \to 0} \big(v_h(x+\veps n_S) + v_h(x-\veps n_S)\big)/2 & \text{if } S \not \subset \p\o_h.
\end{cases}
\]
We make repeated use of inverse estimates, e.g., 
\[
\|\nabla w_h\|_{L^2(T)} \le c_\inv h_T^{-1} \|w_h\|_{L^2(T)}
\]
for a polynomial function $w_h$ of bounded degree on an element $T\in \cT_h$
with diameter $h_T>0$. A scaled trace inequality asserts that for 
a side $S\in \cS_h$ and an adjacent element $T_S \in \cT_h$ we have
\[
c_\trace^{-1} \|w\|_{L^2(S)}^2 \le  h_{T_S}^{-1} \|w\|_{L^2(T_S)}^2 + h_{T_S} \|\nabla w\|_{L^2(T_S)}^2
\]
for all $w\in H^1(T)$. For polynomial functions the second term on the right-hand side
can be omitted at the expense of a larger constant~$c_\trace$. 


\section{Abstract results} \label{sec:abstr_res}
We restrict to the most relevant cases  $0\le \s<1$ and
consider the plate bending problem defined by the energy functional 
\[
I(v) = \frac{\s}{2} \int_\o |\Delta v|^2 \dv{x} + \frac{1-\s}{2} \int_\o |D^2 v|^2 \dv{x}
\]
on the space $V = H^2(\o) \cap H^1_0(\o)$. For a sequence of 
convex subdomains $\oh \subset \o$ and function spaces $V_h \subset L^2(\oh)$
the approximating problems are defined via the functionals
\[
I_h(v_h) = \frac{\s}{2} \int_\oh |\Delta_h v_h|^2 \dv{x} + \frac{1-\s}{2} \int_\oh |D_h^2 v_h|^2 \dv{x},
\]
with a linear operator $D_h^2:V_h \to L^2(\oh)$ that approximates the
Hessian in a sense specified below. The discrete Laplace operator $\Delta_h$ is assumed
to be given by the trace of $D_h^2$. We assign the value $+\infty$ 
to the functionals $I$ and $I_h$ for functions in $L^2(\o)$ not belonging
to $V$ and $V_h$, respectively. We always
extend functions and derivatives defined in $\oh$ trivially to functions defined in $\o$
and impose the following conditions on the approximating problems: 
\begin{itemize}
\item[(S1)] {\em Uniform $H^1_0$-coercivity:} There exists a family of 
operators $\cJ_h : V_h \to H^1_0(\oh)$ such that if $I_h(v_h)$ is 
bounded then $\cJ_hv_h$ is bounded in $H^1_0(\o)$ and $v_h-\cJ_h v_h \to 0$ in $L^2(\o)$. 
\item[(S2)] {\em Stability of $D_h^2$:} If $v_h \wto v$ in $L^2(\o)$ and $D_h^2 v_h \wto \psi$ in 
$L^2(\o)$ then we have that $v\in H^2(\o)$ with $D^2 v = \psi$. 
\item[(S3)] {\em Interpolation in $V_h$:} There exist linear operators
$\cI_h: H^3(\o)\cap H^1_0(\o) \to V_h$ with $D_h^2 \cI_h v \to D^2 v$ 
in $L^2(\o)$ for every $v\in H^3(\o)\cap H^1_0(\o)$.
\end{itemize}

\begin{proposition}[Sufficient conditions]\label{prop:gen_conv}
Assume that conditions~(S1)-(S3) are satisfied. 
We then have that $I_h \to I$ as $h\to 0$ in the sense of $\G$ convergence with respect
to strong convergence in $L^2(\o)$. 
\end{proposition}

\begin{proof}
(i) We first establish the lim-inf inequality. For this, let $(v_h)_{h>0} \subset L^2(\o)$
be a sequence of functions $v_h\in V_h$ with $v_h \to v$ and, without loss of generality,
$I_h(v_h) \le c$. By~(S1) we then have, after extension by zero,
that $(\cJ_h v_h)_{h>0}$, and $(D_h^2 v_h)_{h>0}$ are bounded sequences 
in $H^1_0(\o)$ and $L^2(\o)$ with weak limits $v \in H^1_0(\o)$, and 
$\psi \in L^2(\o;\R^{2\times 2})$ for a suitable subsequence. The assumed consistency 
properties in~(S1) and~(S2) yield that $\psi = D^2 v$ and $v_h \to v$. In 
particular, we have that $v\in V$, and by weak lower semicontinuity of quadratic
functionals we find that
\[
I(v) \le \liminf_{h\to 0} I_h(v_h).
\]
(ii) To verify a lim-sup inequality we choose $v\in V$. By 
density of functions belonging to $V \cap H^3(\o)$ in the set $V$ stated in 
Theorem~\ref{thm:density} and continuity of $I$, we may assume that $v\in H^3(\o)$. 
Condition~(S3) then guarantees that for $v_h = \cI_h v \in V_h$ we have 
$D_h^2 v_h \to D^2 v$ in $L^2(\o)$ and hence $I_h(v_h) \to I(v)$ as $h \to 0$. 
\end{proof}

\begin{remarks}
(i) If $V_h \subset C(\oo_h)$ then one may choose the $P1$ nodal interpolation
operator $\cI_h^{p1}: C(\oo_h) \to \cS^1_0(\cT_h)$ in condition~(S1) 
provided that $v_h(z)=0$ 
for every $v_h\in V_h$ and $z\in \cN_h\cap \p\o$. For discontinuous methods
node averaging or quasiinterpolation operators can be employed. The 
coercivity is then a discrete version of the Poincar\'e estimate~\eqref{eq:poincare_h2}. \\
(ii) Condition~(S2) is independent of boundary approximations and can
be checked in subdomains compactly contained in $\o$. \\
(iii) The interpolation condition~(S3) requires the discrete boundary condition
to be compatible with the interpolation of restricted functions $v|_{\oh}$ whenever
$v\in H^3(\o)\cap H^1_0(\o)$. In particular, conditions $v_h(z)=0$ can
only be imposed at vertices $z\in \cN_h$ belonging to the boundary of $\o$
and not, e.g., at midpoints of edges. \\
(iv) Our approach of imposing the boundary condition only in the corner points of
$\o_h$ coincides with the methods devised in~\cite{Rann79}. One may
incorporate conditions $\nabla u_h(c_i) \cdot t_i= 0$ with the exact tangents
$t_i$ in corner points $c_i$, $i=0,1,\dots,n_m$, as devised in~\cite{StrFix73}
to improve the accuracy, although
this may lead to difficulties in guaranteeing stability of the discrete 
problems. 
\end{remarks}

For discontinuous Galerkin methods the discretization of the functional~$I$
typically involves stabilizing terms defined by positive semidefinite bilinear forms
$s_h: V_h \times V_h\to \R$. The discrete funtionals are then given by
\[
I_{h,\stab}(v_h) = I_h(v_h) + \frac12 s_h(v_h,v_h).
\]
In this case an additional condition is needed: 
\begin{itemize}
\item[(S4)] {\em Stabilization:} For every $v\in H^3(\o)\cap H^1_0(\o)$ and the
sequence $v_h =\cI_hv$ with $\cI_h$ from (S3) we have $s_h(v_h,v_h) \to 0$ as $h\to 0$. 
\end{itemize}

\begin{proposition}[Stabilization]\label{prop:stabilization}
Assume that conditions~(S1)-(S4) are satisfied. 
We then have that $I_{h,\stab} \to I$ as $h\to 0$ in the sense of $\G$ convergence with respect
to strong convergence in $L^2(\o)$.
\end{proposition}

\begin{proof}
Since $s_h(v_h,v_h) \ge 0$ the first part of the proof of Proposition~\ref{prop:gen_conv}
remains valid. The second part also holds by noting that condition~(S4) applies
to the sequence used in the proof. 
\end{proof}

Fully conforming methods fail to converge so that an inconsistency
in the approximation of the differential operator or the boundary conditions
is necessary. 

\begin{proposition}[Necessary condition]\label{prop:failure}
Assume that $\p\o$ contains a curved part, that $V_h\subset H^2(\oh)\cap H^1_0(\oh)$,
and $D_h^2 = D^2$ for all $h>0$. Then the functionals $I_h$, $h>0$, are 
not $\G$-convergent to $I$ as $h\to 0$.
\end{proposition}

\begin{proof}
Assume that $I_h(v_h) \to I(v)$ for some $v\in V$ and a
sequence $(v_h)_{h>0}$ with $v_h\in V_h$ and $v_h \to v$.
Since $I_h$ is quadratic in $D^2 v_h$, it then follows that $D^2 v_h \to D^2 v$  
and $\Delta v_h \to \Delta v$ in $L^2(\o)$. This however contradicts the
convergence of the representations~\eqref{eq:fl_no_curv_term} 
of $I_h$ to the representation~\eqref{eq:fl_curv_term} of $I$ with boundary integral 
terms provided by Lemma~\ref{la:bdy_repr} unless the boundary of 
$\o$ is piecewise straight or $\p_n v=0$ on $\p\o$. 
\end{proof}

The assumed convexity condition on $\o$ simplifies certain arguments but can be
avoided. 

\begin{remark}
If $\o$ is not convex then by choosing a sufficiently large domain $\widehat{\o}$ 
that contains the approximating domains $\o_h$ one establishes convergence in 
$H^1_0(\widehat{\o})$ and shows that limitting functions are supported in $\overline{\o}$.
\end{remark}

\section{Numerical methods}\label{sec:methods}
We apply the abstract results of the previous section to prototypical finite
element methods for fourth order problems. Throughout this section we assume
that the conditions (T1)-(T2) on the triangulations are satisfied. We always
extend functionals by the value $+\infty$ to the entire set $L^2(\o)$.

\subsection{A conforming method}\label{sec:conf_fe}
For a triangle $T\in \cT_h$ with vertices $z_0,z_1,z_2$ 
the Argyris element is given as the triple $(T,P_5(T),K_T)$ with the set of node 
functionals $K_T$ containing the
functionals $\chi_{i,\a}(v) = \p^\a v(z_i)$ for $i=0,1,2$, $\a\in \N_0^2$ with
$\a_1+\a_2\le 2$, and $\chi_{i,n}=\nabla v(x_{S_i})\cdot n_{S_i}$ associated with
the sides $S_i$, $i=0,1,2$, of $T$ with midpoints $x_{S_i}$ and outer normals $n_{S_i}$.
The element is illustrated in Figure~\ref{fig:elements} and  
leads to the space 
\[
\cS^{5,1}(\cT_h) = \{ v_h \in C^1(\oo_h): v_h|_T \in P_5(T) \mbox{ for all }T\in \cT_h\}.
\]
For subspaces $V_h$ of 
$\cS^{5,1}(\cT_h)$ containing boundary conditions we consider the minimization of 
\[
I_h^\arg(v_h) = \frac{\s}{2} \int_{\o_h} |\Delta v_h|^2 \dv{x} 
+ \frac{1-\s}{2} \int_{\o_h} |D^2 v_h|^2 \dv{x} 
\]
in the set of functions $v_h\in V_h$. 

\begin{proposition}[Failure of convergence]
Let $V_h^{\arg,o} = H^1_0(\o_h)\cap \cS^{5,1}(\cT_h)$ and assume that $\p\o$ contains
a curved part. Then the functionals $I_h^\arg:V_h^{\arg,o}\to \R$ are not $\G$-convergent 
to the functional $I$.
\end{proposition}

\begin{proof}
The result is a direct consequence of Proposition~\ref{prop:failure}.
\end{proof}

Reducing the discrete boundary condition to the vertices on the boundary
leads to correct convergence. 

\begin{proposition}[Correct convergence]\label{prop:conv_arg}
Let 
\[
V_h^\arg = \big\{ v_h \in \cS^{5,1}(\cT_h): v_h(z) = 0 \mbox{ for all }z\in \cN_h \cap \p\oh \big\}.
\]
Then the functionals $I_h^\arg: V_h^\arg\to \R$ are $\G$-convergent to $I$ as $h\to 0$.
\end{proposition}

\begin{proof}
We verify the conditions of Proposition~\ref{prop:gen_conv} to deduce the result. \\
(i) Given $v_h\in V_h^\arg$, we have that $\cI_h^\po v_h \in \cS^1_0(\cT_h)$ 
so that an integration by parts leads to 
\[
\|\nabla  \cI_h^\po v_h\|^2 
= - \int_{\o_h} \cI_h^\po v_h \Delta v_h \dv{x} + 
\int_{\o_h} \nabla \cI_h^\po v_h \cdot \nabla (\cI_h^\po v_h - v_h) \dv{x}.
\]
A Poincar\'e inequality and a nodal interpolation estimate imply that
\[
\|\nabla \cI_h^\po v_h\| \lesssim \|\Delta v_h\| + h \|D^2 v_h\|,
\]
which guarantees the uniform coercivity property~(S1). \\
(ii) Since $D_h^2 = D^2$ the stability requirement~(S2) of the discrete
Hessian is trivially satisfied. \\
(iii) A modification of the canonical interpolation operator 
$\cI_h^\arg: H^4(\o) \to \cS^{5,1}(\cT_h)$ defined by the node functionals is required 
to interpolate functions in $H^3(\o)$. 
Given $v\in H^3(\o)\cap H^1_0(\o)$ we determine polynomials $p_T\in P_5(T)$ via the conditions 
$\chi_{i,\a} (p_T - v) = 0$, if $|\a|\le 1$, and $\chi_{i,n}(p_T-v) = 0$,  and 
\[
\tchi_{i,\a}(p_T) = \frac{1}{|\o_{z_i}|} \int_{\o_{z_i}} \p^\a v \dv{x},
\]
if $|\a|=2$ and for $i=0,1,2$, where $\o_{z_i}$ is the union of elements in $\cT_h$
that contain $z$. The modified interpolant $\tcI_h^\arg v \in V_h$ of $v$ is then 
defined via $(\tcI_h^\arg v) |_T = p_T$ for all $T\in \cT_h$. Owing to
$v|_{\p\o}=0$ we have that $v_h(z)=0$ for all $z\in \cN_h\cap \p\o_h$, so that
$\tcI_h^\arg v\in V_h^\arg$. Since $\tcI_h^\arg$
reproduces quadratic functions, an interpolation estimate follows from the Bramble--Hilbert
lemma, which implies condition~(S3). 
\end{proof}

\begin{figure}[htb]
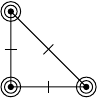 \hspace{1.4cm}
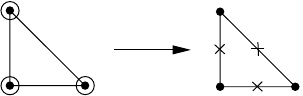
\caption{\label{fig:elements} Schematical description of the Argyris element (left)
and the discrete Kirchhoff triangle with cubic and quadratic polynomial spaces (right).}
\end{figure}

\subsection{A nonconforming element}\label{sec:nconf_fe}
The discrete Kirchhoff triangle
uses $C^0$ conforming scalar and vectorial spaces $\cS^{3,\dkt}(\cT_h) \subset H^1(\o_h)$ and 
$\cS^2(\cT_h)^2 \subset H^1(\o_h;\R^2)$, respectively, and a discrete gradient operator 
\[
\nabla_h : \cS^{3,\dkt}(\cT_h) \to \cS^2(\cT_h)^2,
\]
cf. Figure~\ref{fig:elements}, that approximates the weak gradient.  
With this, we define the operators 
$D_h^2 v_h = \nabla \nabla_h v_h$ and $\Delta_h v_h = \diver \nabla_h v_h$ for $v_h\in  \cS^{3,\dkt}(\cT_h)$
and the discrete functionals  
\[
I_h^\dkt(v_h) = \frac{\s}{2} \int_{\o_h} |\Delta_h v_h|^2 \dv{x} 
+ \frac{1-\s}{2} \int_{\o_h} |D_h^2 v_h|^2 \dv{x},
\]
defined for $v_h \in V_h^\dkt$ with
\[
V_h^\dkt = \big\{v_h \in \cS^{3,\dkt}(\cT_h): v_h(z)=0 \text{ for all } z \in \cN_h\cap \p\o_h\big\}.
\]
Letting $\nabla_\cT$ and $D_\cT^2$ denote the elementwise application of the gradient and the Hessian,
we have the equivalence of the seminorms $\|D_\cT^2  v_h\|$ and $\|D_h^2 v_h\|$, 
and for all $v_h \in \cS^{3,\dkt}(\cT_h)$ and $v\in H^3(\o)$ the estimates, cf.~\cite{Brae07,Bart15-book},
\[\begin{split}
\|\nabla_h \cI_h^{3,\dkt} v -\nabla v\|
+ h \|D_h^2 \cI_h^{3,\dkt} v - D^2 v\| &\lesssim h^2 \|D^3 v\|, \\
\|\nabla_h v_h  - \nabla v_h\| &\lesssim h \|D_\cT^2 v_h\|.
\end{split}\]
The following proposition shows that imposing the boundary condition only in the
boundary nodes is sufficient to avoid incorrect convergence.

\begin{proposition}[Correct convergence]
Let 
\[
V_h^\dkt = \big\{ v_h \in \cS^{3,\dkt}(\cT_h): v_h(z) = 0 \mbox{ for all }z\in \cN_h \cap \p\o \big\}.
\]
Then the functionals $I_h^\dkt$ restricted to $V_h^\dkt$ are $\G$-convergent to $I$. 
\end{proposition}

\begin{proof}
We verify the conditions of Proposition~\ref{prop:gen_conv} to deduce the result. \\
(i) As in the proof of Proposition~\ref{prop:conv_arg} we find that
\[
\|\nabla \cI_h^\po v_h \|^2 = - \int_\oh \cI_h^\po v_h \cdot \Delta_h v_h \dv{x} + 
\int_\oh \nabla \cI_h^\po v_h \cdot (\nabla \cI_h^\po v_h - \nabla_h v_h )  \dv{x},
\]
which implies the condition~(S1) after application of a Poincar\'e inequality, i.e.,  
\[
\|\nabla \cI_h^\po v_h \| \lesssim \|\Delta_h v_h\| + h \|D_\cT^2 v_h\|.
\]
(ii) Let $(v_h)_{h>0}$ be a sequence of functions $v_h\in V_h$ 
with $v_h \to v$ in $L^2(\o)$ as $h\to 0$. For every compactly supported
function $\phi \in C^\infty_0(\o;\R^{2\times 2})$ we then find with the approximation
properties of $\nabla_h$ that 
\[
\int_\o D_h^2 v_h : \phi \dv{x} = -\int_\o \nabla_h v_h \cdot \Diver \phi \dv{x} 
\to -\int_\o \nabla v \cdot \Diver \phi \dv{x}
\]
as $h \to 0$, and hence $v\in H^2(\o)$ with $D^2 v = \lim_{h\to 0} D_h^2 v_h$, which
provides condition~(S2). \\
(iii) The canonical nodal interpolation operator 
$\cI_h^\dkt : H^2(\o_h)\to  \cS^{3,\dkt}(\cT_h)$ defined by $\cI_h^\dkt v(z) = v(z)$ 
and $\nabla \cI_h^\dkt v(z) = \nabla v(z)$ has the features needed to guarantee~(S3). 
\end{proof}

\begin{remark}
By following the arguments of this subsection, convergence of discretizations based 
on the Morley element can be established, cf., e.g.,~\cite{Rann79b}.
\end{remark}

\subsection{A dG method}\label{sec:dg_fe}
We consider a simple discontinuous Galerkin method that does not fit exactly
into the general framework provided by Proposition~\ref{prop:gen_conv}
but the convergence analysis only requires minor modifications. The method
uses the space 
\[
V_h^\dg = \cL^\ell(\cT_h)
\]
of elementwise polynomials of degree at most $\ell \ge 0$ and the functionals 
\[
I_h^\dg(v_h) = \frac12 a_h(v_h,v_h) + \frac12 s_h(v_h,v_h).
\]
Here, $a_h$ encodes a discretization of the weak form of the differential
operator defined by $I$ and $s_h$ are stabilizing bilinear forms. For 
simplicity, we consider the case $\s=0$ so that the elastic energy is
a multiple of the integral of the squared norm of the Hessian. 
In the following and in contrast to the previous subsection, the symbols
\[
\nabla_h, \ D_h^2, \ \Delta_h
\]
denote here the elementwise application of the indicated differential operators.
A bilinear form resulting from an elementwise integration by
parts and a consistent symmetrization is given by 
\[\begin{split}
a_h(v_h,w_h) 
&= (D_h^2 v_h,D_h^2 w_h) \\
&\quad + (\aver{\pnnh v_h},\jj{\nh w_h})_{\cSh\setminus \p\o_h} + (\aver{\pnnh w_h},\jj{\nh v_h})_{\cSh\setminus \p\o_h}  \\
&\quad - (\aver{\pnDh v_h},\jj{w_h})_\cSh - (\aver{\pnDh w_h},\jj{v_h})_\cSh .
\end{split}\] 
With factors $\g_0,\g_1>0$ and the length function $h_\cS|_S = h_S$ of sides, 
the stabilizing bilinear form is given by 
\[\begin{split}
s_h(v_h,w_h) &= \g_0 ( h_\cS^{-3} \jj{v_h},\jj{w_h})_\cSh \\
&\quad + \g_1  ( h_\cS^{-1} \jj{\nh v_h}, \jj{\nh w_h})_{\cSh\setminus \p\o_h}.
\end{split}\]
The convergence analysis of the functionals makes repeated use of the scaled
trace inequality in the form $\|h_\cS^{1/2} v_h\|_{L^2(\cSh)} \lesssim \|v_h\|$
for every $v_h\in V_h^\dg$.

{\em Equicoercivity in $H^1_0(\o)$.}  We define the discrete norm
\[
\|v_h\|_\dg^2 = \|D_h^2 v_h\|^2 + s_h(v_h,v_h)
\]
and note that with trace inequalities and inverse estimates 
one obtains for $\g_0,\g_1>0$ sufficiently large that there exists $\a>0$ with
\[
a_h(v_h,v_h) \ge \a \|v_h\|_\dg^2.
\]
Given $v_h\in V_h^\dg$ we consider $\cJ_h^\av v_h\in H^1_0(\o_h)$ and note via
elementwise integration by parts, H\"older, Poincar\'e, and trace 
inequalities, that
\[\begin{split}
\|\nabla & \cJ_h^\av v_h\|^2 
= \int_\oh \nh v_h \cdot \nabla \cJ_h^\av v_h \dv{x} 
  + \int_\oh \nh (\cJ_h^\av v_h - v_h) \cdot \nabla \cJ_h^\av v_h \dv{x} \\
&\lesssim \big(\|\Delta v_h\|  + \|h_\cS^{-1/2} [\nh v_h]\|_{\cSh} 
+ \|\nh (\cJ_h^\av v_h -v_h)\| \big) \| \nabla \cJ_h^\av v_h\| .
\end{split}\]
Incorporating the estimate for the node averaging operator, we deduce that
\[
\|\nabla \cJ_h^\av v_h\| \lesssim  \|\Delta v_h\| + \|h_\cS^{-1/2} \jj{\nh v_h}\|_{\cSh} 
+ \|h_\cS^{-1/2} \jj{v_h}\| \lesssim \|v_h\|_\dg.
\]
This establishes the uniform coercivity in $H^1_0(\o)$. 

{\em Interpolation and stabilization in $V_h$.}
Given a function $v\in H^3(\o)\cap H^1_0(\o)$ we consider the quadratic Lagrange
interpolant $v_h = \cI_h^{2,0} v \in V_h^\dg \cap C(\overline{\o}_h)$, assuming that $V_h$ contains
quadratic polynomials, i.e., $\ell \ge 2$. For every $T\in \cT_h$ we have for
$r=0,1,2$ that
\[
\|D^r (v_h-v)\|_{L^2(T)} \lesssim h_T^{3-r} \|D^3 v\|_{L^2(T)}.
\]
In combination with the trace inequality one obtains that
\[
\|h_\cS^{-1/2} \jj{\nh v_h}\|_{L^2(\cS_h\setminus \p\oh)} 
= \|h_\cS^{-1/2} \jj{\nh (v_h-v) }\|_{L^2(\cS_h\setminus \p\oh)}  \lesssim h \|D^3 v\|.
\]
On inner sides we have $\jj{v_h}=0$ while for the union of boundary sides $S\subset \p\oh$ we have
\[
\|\jj{v_h}\|_{L^2(\p\o_h)} \lesssim \|v\|_{L^\infty(\p\o_h)} \lesssim h^2 \|\nabla v\|_{L^\infty(\o)},
\]
where we used that $v|_{\p\o}=0$ and for every $x\in S$ there exists $x'\in \p\o$ with 
$|x-x'|\le ch^2$. Finally, we note that
\[\begin{split}
(\jj{\pnnh v_h} ,\jj{\nh v_h})_S  &\le \|D_h^2 v_h\|_S \|\jj{\nh v_h}\|_S \\
&\lesssim h_S^{-1/2} \|D_h^2 v_h\|_{L^2(T_S)} h_S^{3/2} \|D^3 v \|_{L^2(T_S)} \\
&\lesssim h_S \|D^3 v\|_{L^2(T_S)} \|D_h^2v_h\|_{L^2(T_S)}.
\end{split}\]
Altogether, sums of squared terms related to sides in $a_h$ vanish as $h\to 0$, 
we have $s_h(v_h,v_h)\to 0$, and in particular for $h\to 0$ that
\[
I_h^\dg(v_h) \to  I(v).
\]

{\em Hessian stability.} To establish a liminf inequality, we follow the arguments of~\cite{BoNoNt21}.
Therein, a discrete Hessian  matrix $H_h(v_h)$ is constructed for every $v_h\in V_h$ via a
lifting operation, which satisfies the relation 
\[\begin{split}
(H_h(v_h),D_h^2 w_h) & = (D_h^2 v_h,D_h^2 w_h) \\ 
& \quad - (\jj{\nh v_h},\aver{\pnnh w_h})_\cSh + ( \jj{v_h},\aver{\pnDh w_h})_\cSh,
\end{split}\]
and obeys the inequality
\[
\| H_h(v_h)\| \le a_h(v_h,v_h).
\]
Moreover, if $(v_h)_{h>0}$ is a sequence of functions $v_h\in V_h$ such that 
$a_h(v_h,v_h)$ is uniformly bounded and $v_h\to v$ in $L^2(\o)$, 
then there exists $v\in H^2(\o)$ with 
\[
(H_h(v_h),\phi) \to (D^2 v ,\phi)
\]
for all $\phi \in C_0^\infty(\o;\R^{2\times 2})$. This convergence is established 
in~\cite{BoNoNt21} for a fixed domain that is accurately triangulated. However, 
since test functions are compactly supported
the result carries over to the case of approximated domains. 

\medskip

The equicoercivity in $H^1_0(\o)$, the interpolation property, and the Hessian consistency
imply the following $\G$ convergence result. 

\begin{proposition}[Correct convergence]
Let $V_h^\dg$ contain the set of elementwise quadratic, globally continuous polynomials,
and assume that $\g_0,\g_1>0$ are sufficiently large.
Then the functionals $I_h^\dg:V_h^\dg\to \R$ are $\G$-convergent to $I$. 
\end{proposition}

\begin{remark}
Local discontinuous Galerkin methods as discussed in~\cite{BGNY23}
replace the Hessian in the functional $I$ by the approximation $H_h$ and thereby fit into the abstract 
framework of Proposition~\ref{prop:gen_conv} for arbitrary choices of the parameters $\g_0,\g_1>0$.
\end{remark}

\section{Numerical experiments}\label{sec:num_ex}
We verify the theoretical results via numerical experiments for the setting
considered in~\cite{BabPit90}. All experiments were carried out using elementary
realizations in \textsc{Matlab} as in~\cite{Bart15-book,Bart16-book}. 

\begin{example}[Simple support on unit disk]\label{ex:bab_ex}
Let $\o = B_1(0)$ and $f=1$. Then the minimizer $u\in V=H^2(\o)\cap H^1_0(\o)$
for the functional $I$ satisfies $\Delta^2 u = f$, $u=\Delta u - (1-\s)\p_n u =0$
and is given by
\[
u(x) = \frac{(5+\s)-(6+2\s) |x|^2 + (1+\s)|x|^4}{64(1+\s)}.
\]
The solution obtained as a limit of an operator splitting on polygonal domains
solves $\Delta^2 u_\infty = f$ and $u_\infty = \Delta u_\infty =0$ in $\p\o$ and
is given by 
\[
u_\infty(x) = \frac{3}{64} - \frac{1}{16}|x|^2 + \frac{1}{64} |x|^4.
\]
We always set $\s=0$.
\end{example}

We test the methods analyzed above on triangulations that are obtained via
uniform refinements of triangulations of a square via the correction 
$z\mapsto \sqrt{2}|z|_\infty z/|z|_2 $, cf. Figure~\ref{fig:triangs}.  
Corresponding nodal interpolants of the functions $u$ and $u_\infty$ are shown in 
Figure~\ref{fig:sol_and_wrong}.

\begin{figure}[htb]
\includegraphics[width=.2\linewidth]{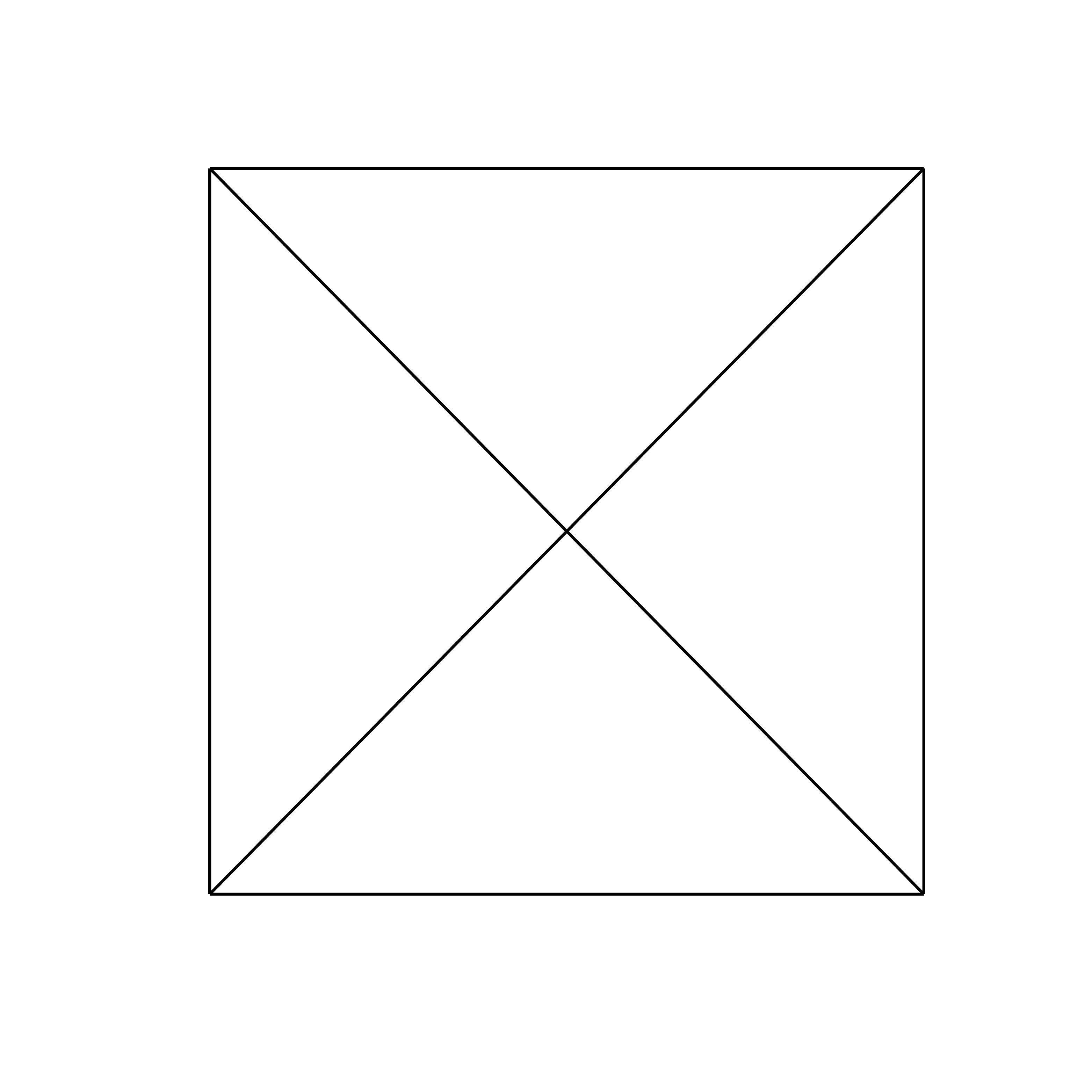} \hspace*{2mm}
\includegraphics[width=.2\linewidth]{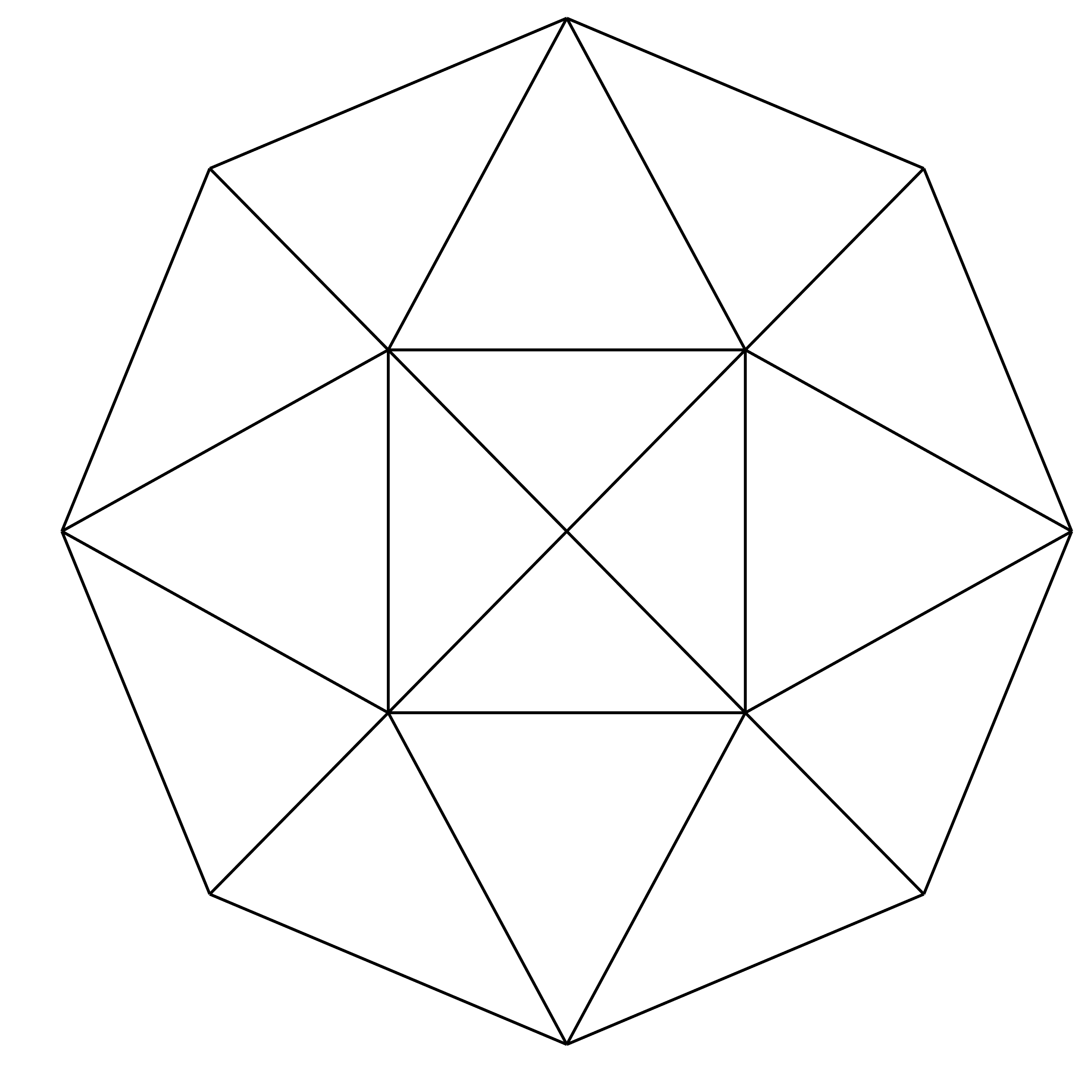} \hspace*{2mm}
\includegraphics[width=.2\linewidth]{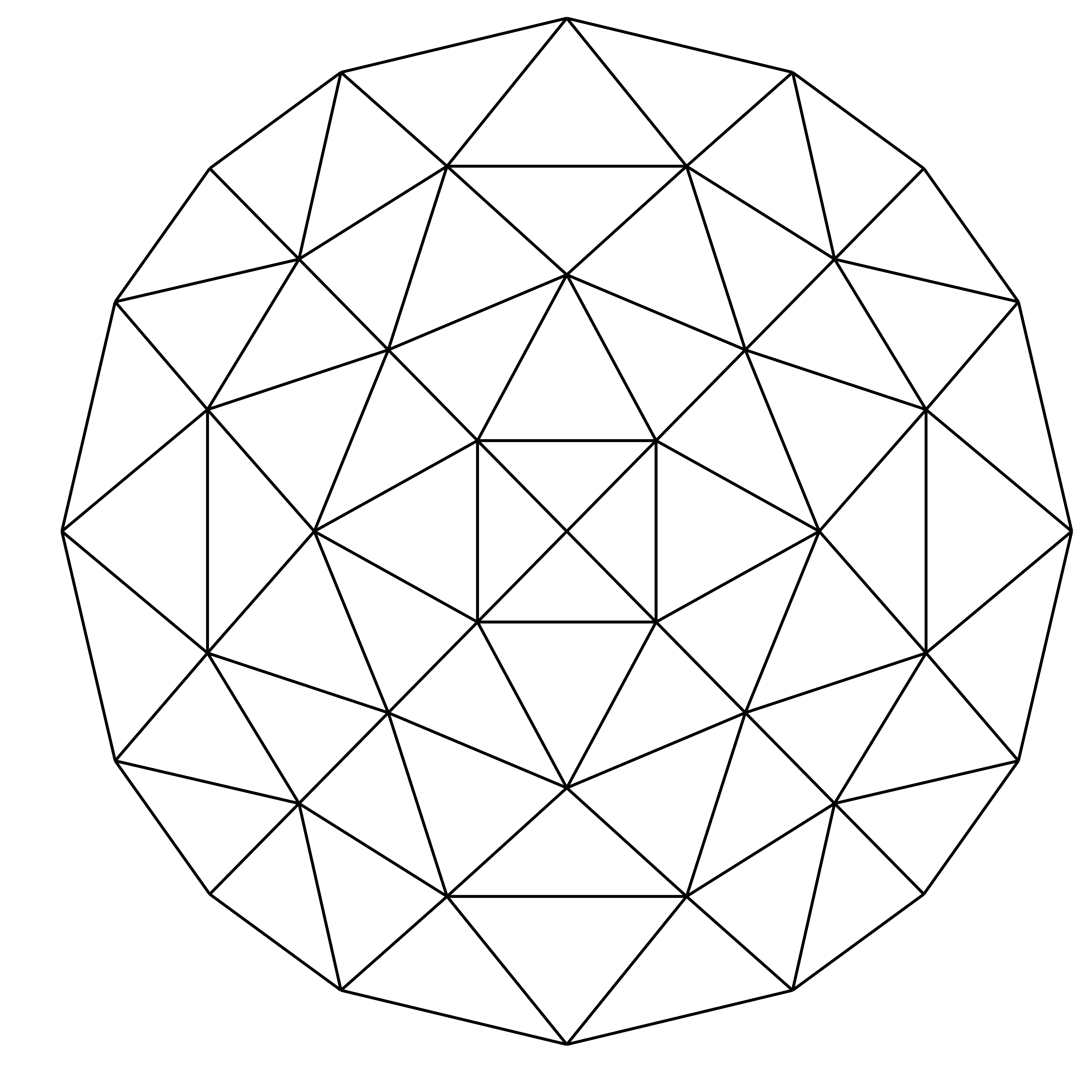} \hspace*{2mm}
\includegraphics[width=.2\linewidth]{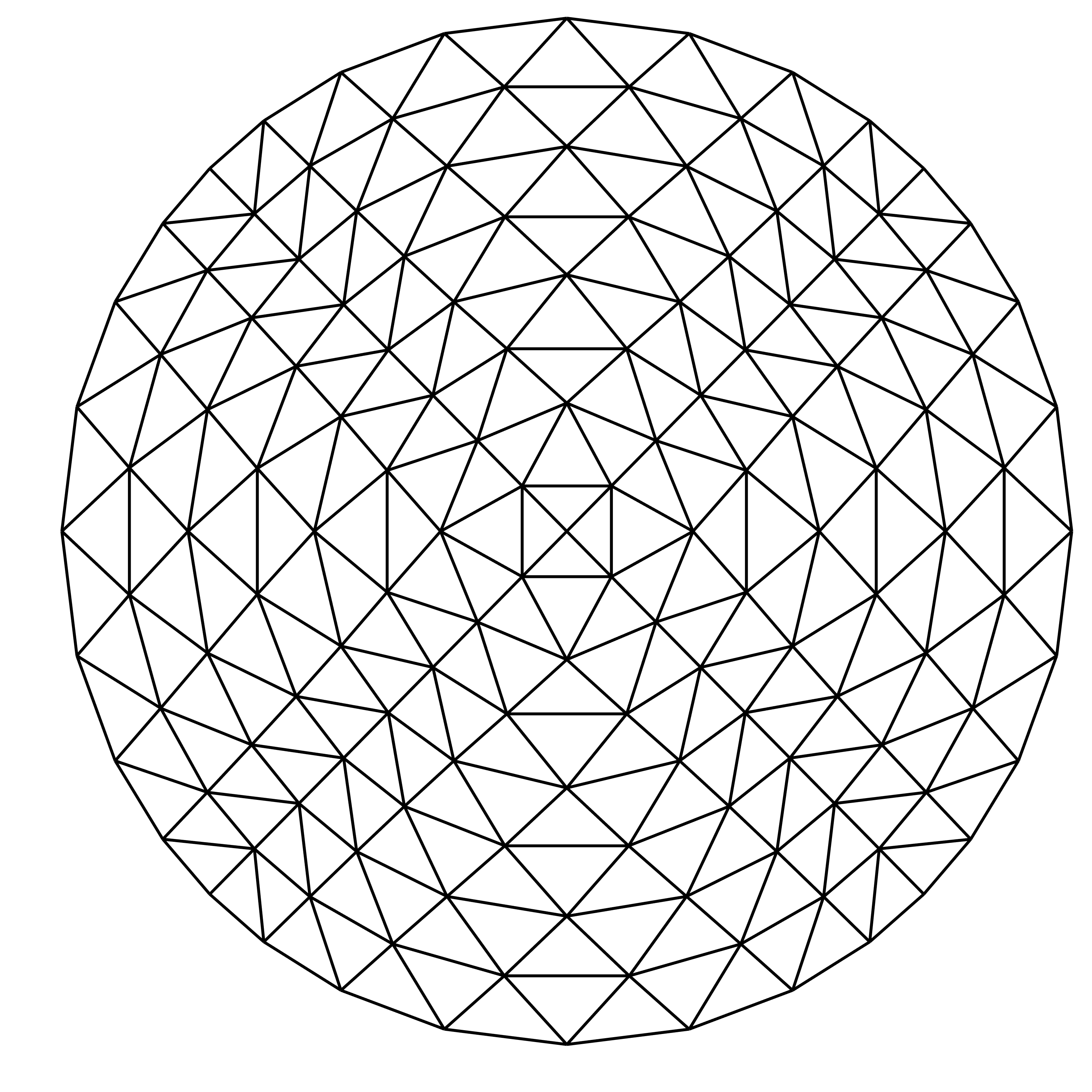}
\caption{\label{fig:triangs} Triangulations obtained from refinements and
subsequent corrections of triangulations of a square.}
\end{figure}

\subsection{Conforming method}
Figure~\ref{fig:argyris_full_vs_nodal} shows finite element solutions 
using the Argyris element for a full realization of the boundary 
conditions as well as their reduced, nodal treatment as introduced in
Section~\ref{sec:conf_fe}. We see that the
maximal values at the center of the disk differ substantially and
only the reduced treatment of the boundary conditions leads to correct
approximations. The effect of the reduced treatment is visualized in 
Figure~\ref{fig:reduced_bdy}, where we observe that small arcs form along
boundary sides that provide the right flexibility for the approximations
to attain the correct maximal values. 

\begin{figure}[htb]
\includegraphics[width=4cm]{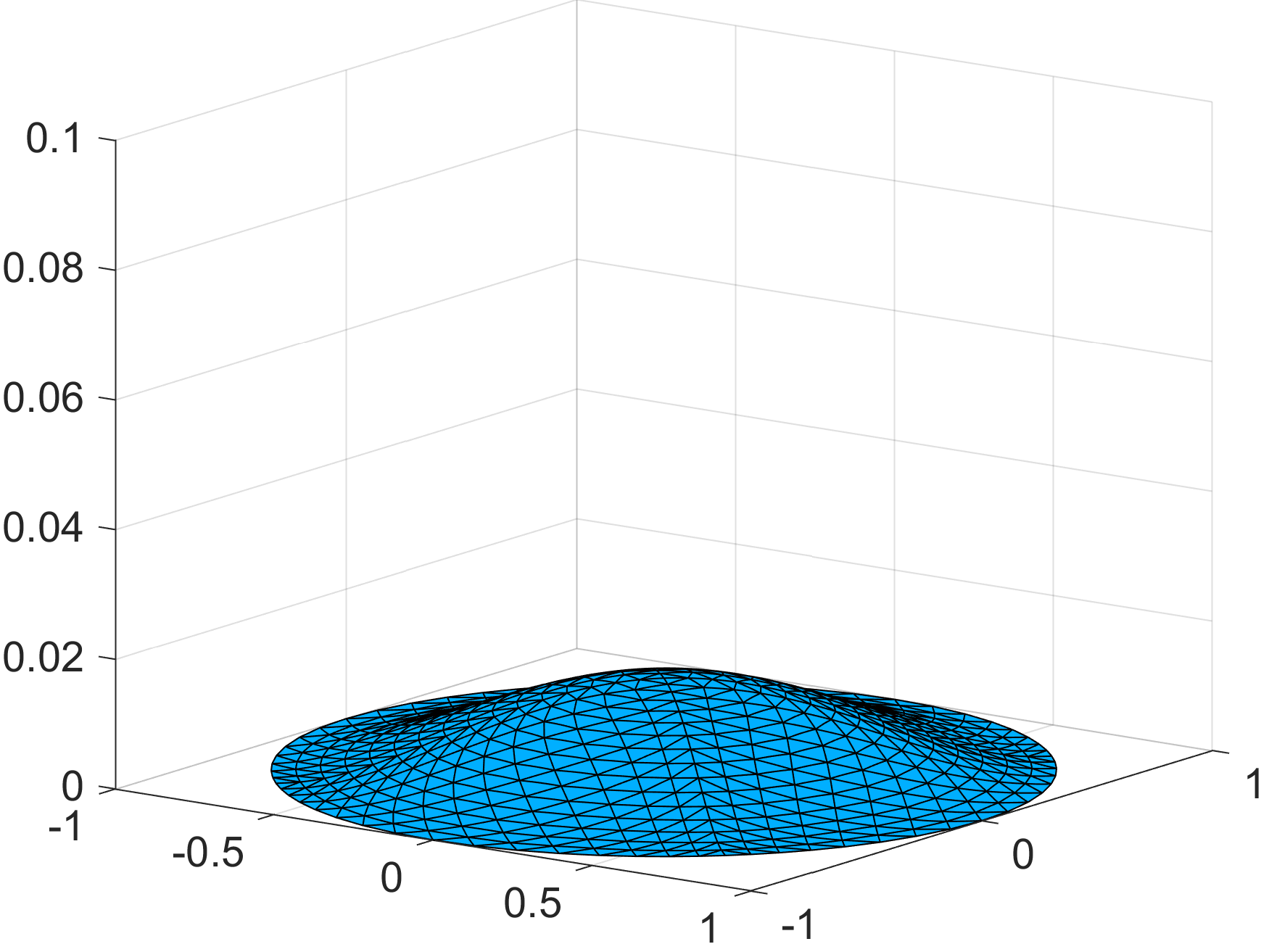} \hspace*{4mm}
\includegraphics[width=4cm]{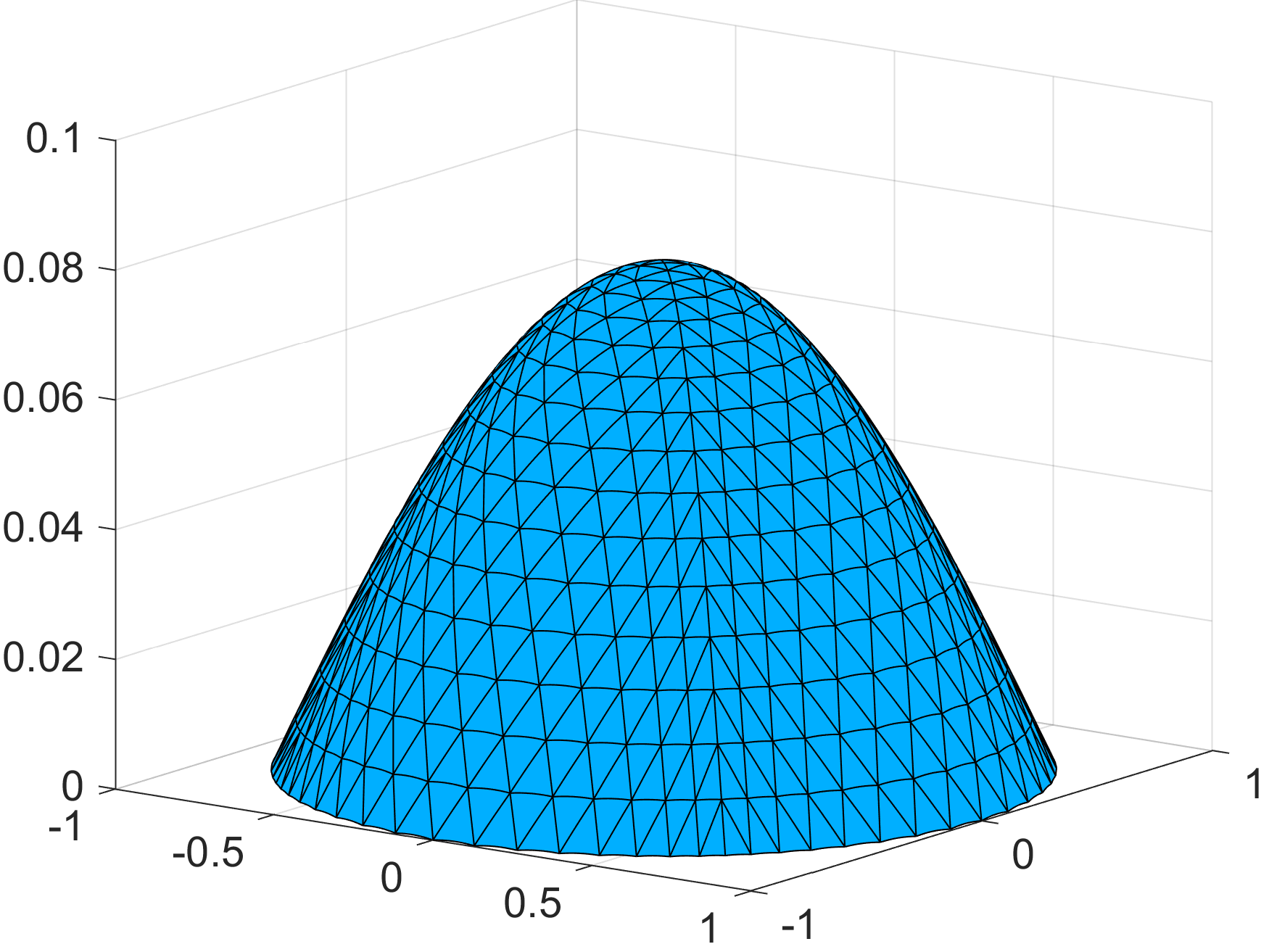} 
\caption{\label{fig:argyris_full_vs_nodal} Finite element approximations
obtained with the Argyris element and a full (left) and reduced (right)
realization of the simple support boundary conditions.}
\end{figure}

\begin{figure}[htb]
\includegraphics[width=4cm]{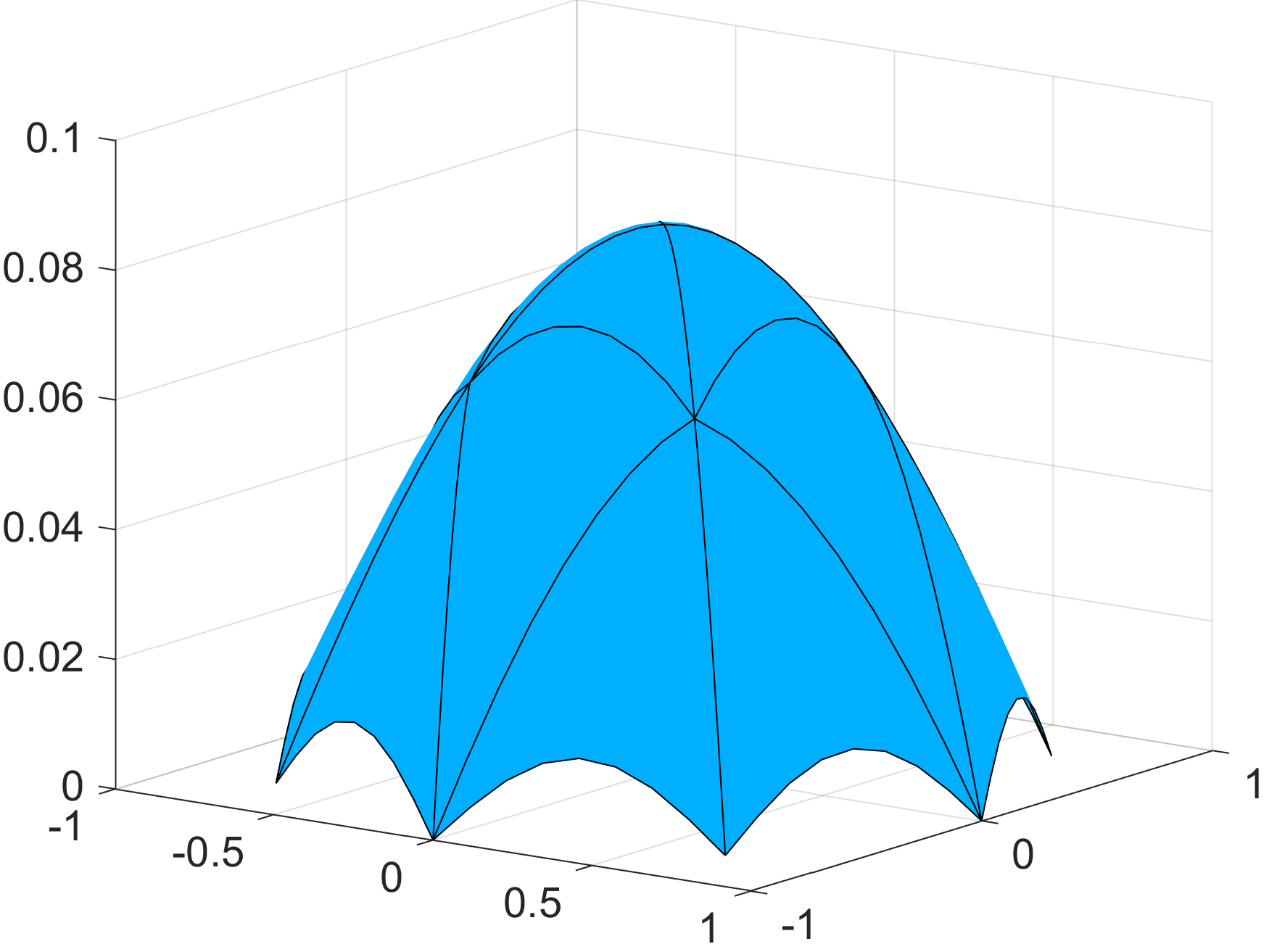} \hspace*{4mm}
\includegraphics[width=4cm]{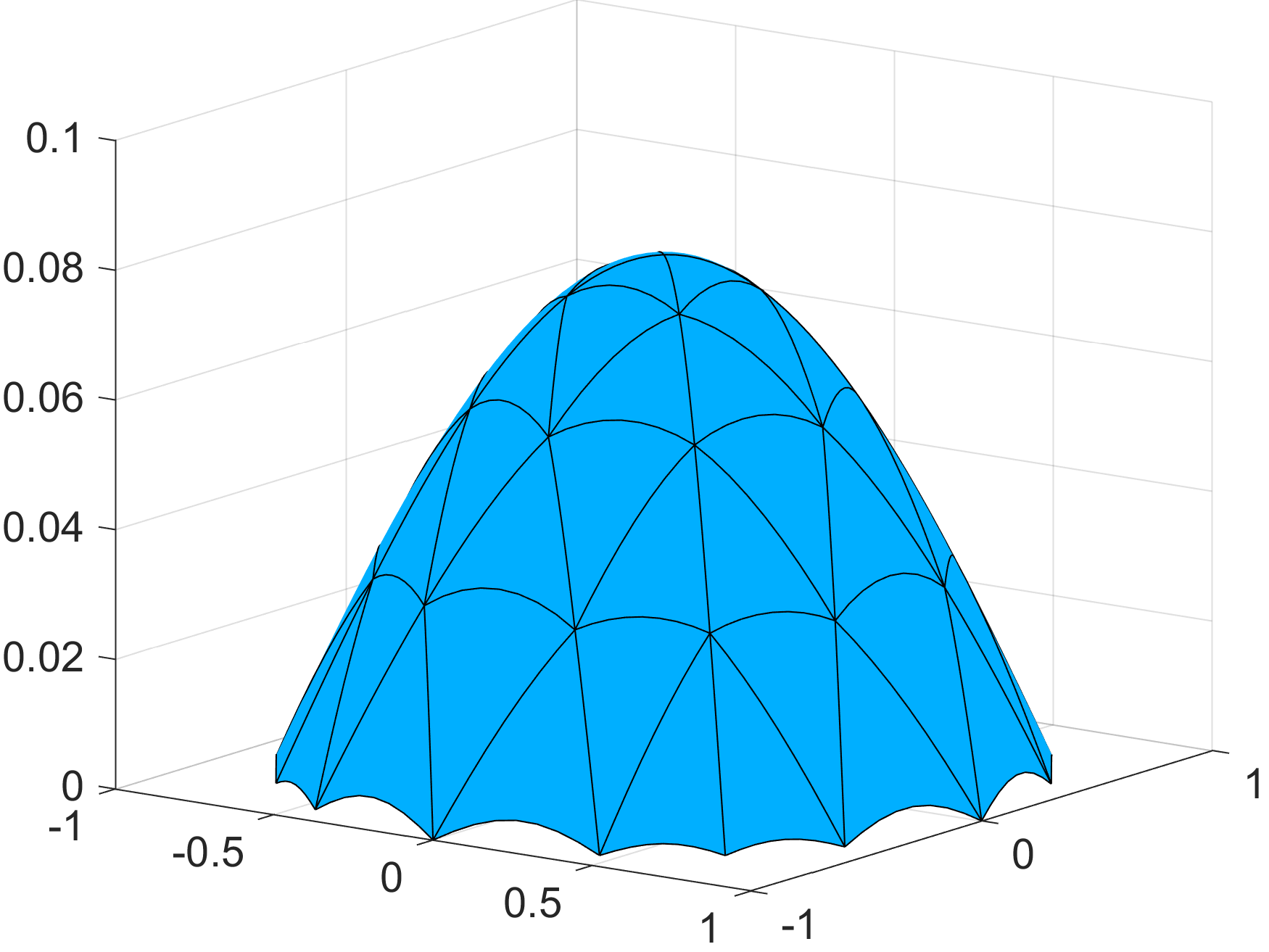} 
\caption{\label{fig:reduced_bdy} Coarse finite element approximations
obtained with the Argyris element and simple support boundary conditions
imposed in the corner points.}
\end{figure}

\subsection{Nonconforming method}
Approximating the solution of the model problem using the discrete Kirchhoff
element as in Section~\ref{sec:nconf_fe} leads to the approximation shown in the 
left plot of Figure~\ref{fig:dkt_and_dg}. The numerical solution accurately 
approximates the exact solution. 

\subsection{DG method}
The discontinuous Galerkin method analyzed in Section~\ref{sec:dg_fe} with
parameters $\g_0=\g_1=10$ provides the numerical solution shown in 
the right plot of Figure~\ref{fig:dkt_and_dg}. As in the case of the
nonconforming approximation, the maximal value accurately approximates the 
maximal value of the exact solution. 

\begin{figure}[htb]
\includegraphics[width=4cm]{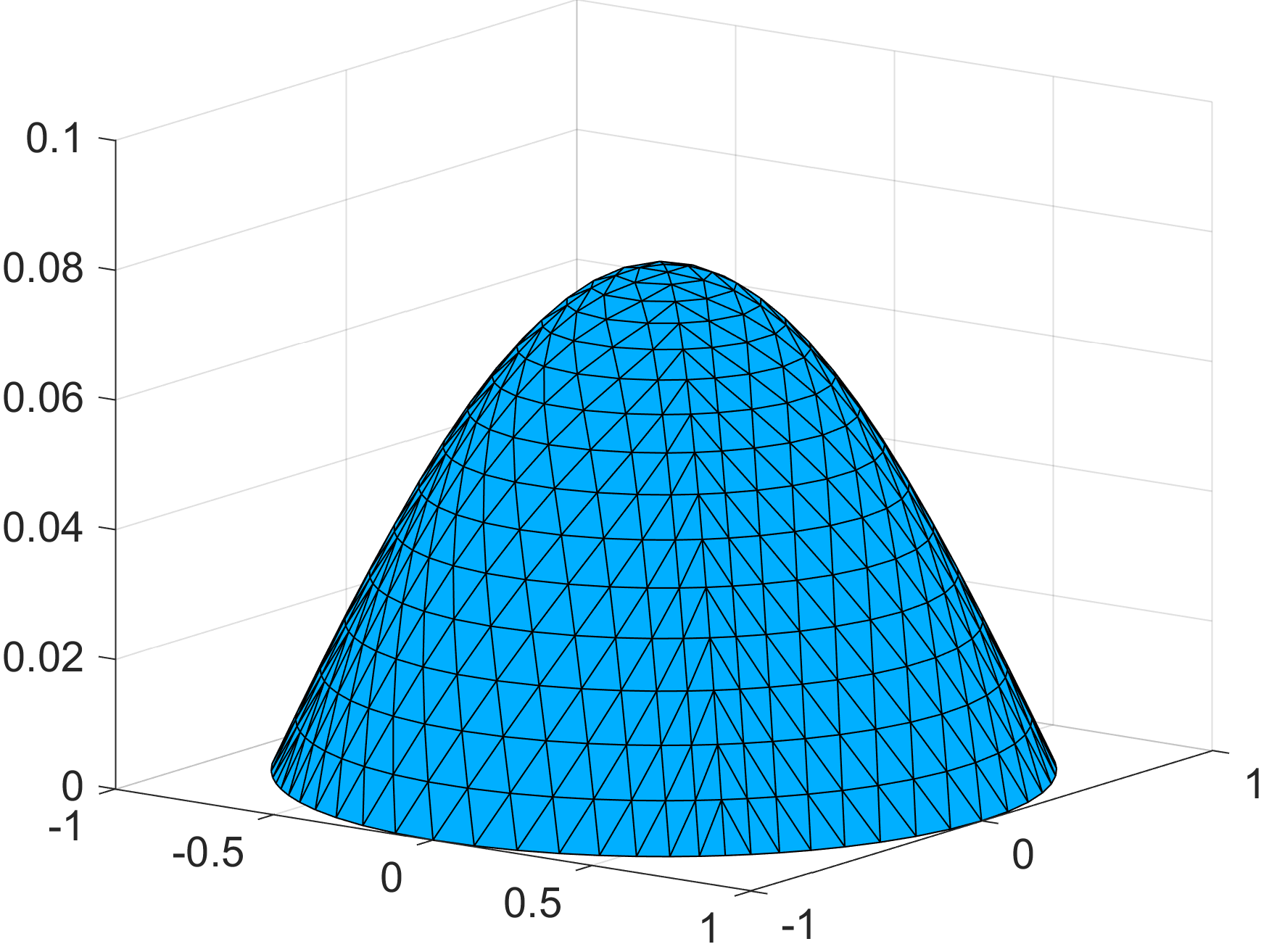} \hspace*{4mm}
\includegraphics[width=4cm]{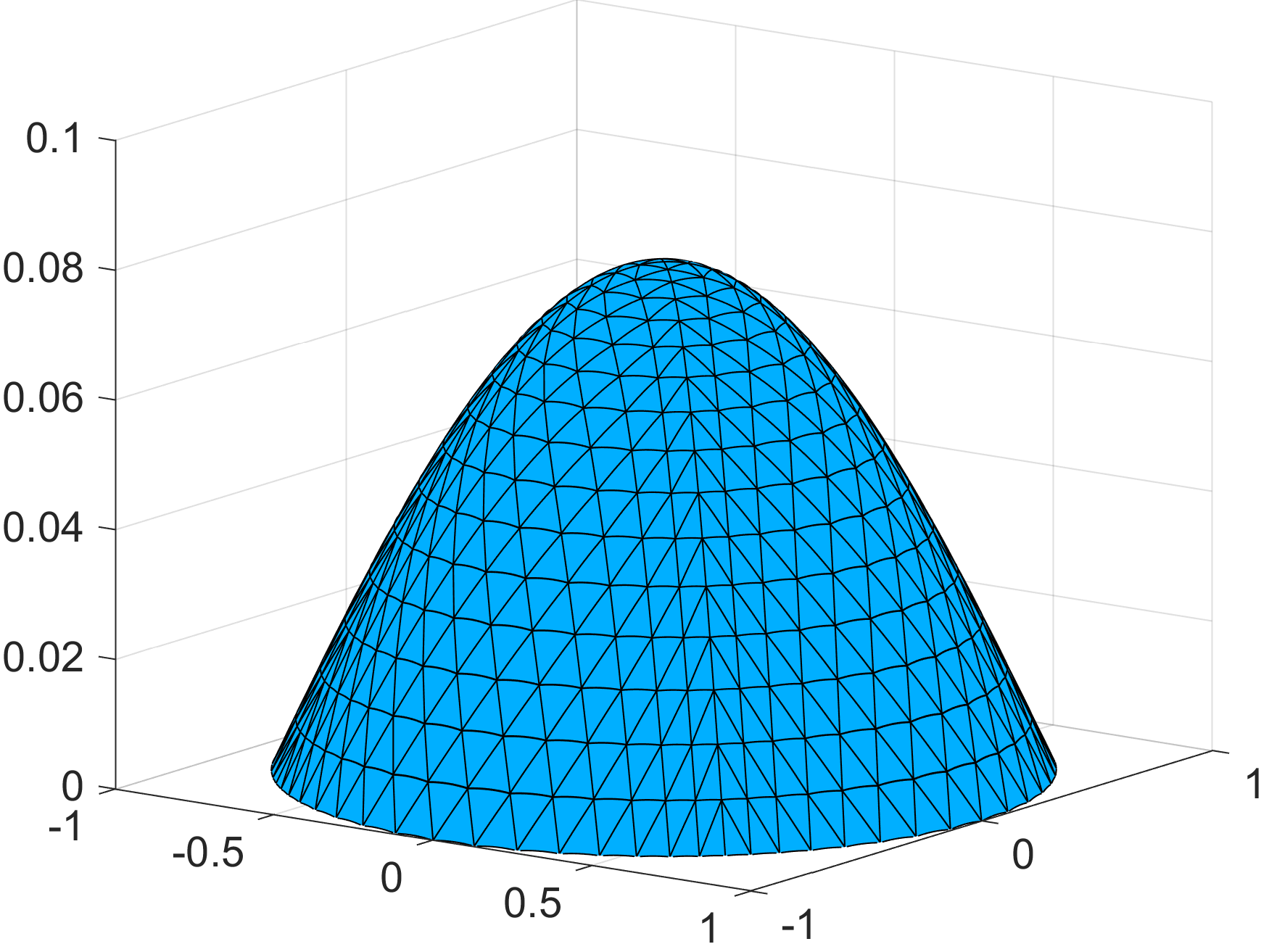} 
\caption{\label{fig:dkt_and_dg} Finite element approximations
obtained with the discrete Kirchhoff element (left) and a discontinuous 
Galerkin method (right).}
\end{figure}

\subsection{Convergence rates}
The plots shown in Figures~\ref{fig:midpoint} and~\ref{fig:h2-error} display
the convergence behavior of the errors in approximating the midpoint
value $u(0)$ and with respect to discrete $H^2$ norms, i.e., the
error quantities 
\[
\d^{\rm mp}_h = |u_h(0)-u(0)|, \quad
\d^{H^2}_h = \|u_h - \cI_h u\|_{a_h},
\]
where $a_h$ is the discrete bilinear form defined by the different
numerical methods and $\cI_h$ is the corresponding nodal interpolant. 
We observe from Figure~\ref{fig:midpoint} that the approximations obtained 
with the Argyris method and
reduced boundary condition (Argyris, nodal support), with the discrete Kirchhoff method (DKT),
and discontinuous Galerkin methods with quadratic polynomials and
linear (DG, $P1$ boundary) as well as isoparametric quadratic (DG, $P2$ boundary)
approximations of the domain boundary provide highly accurate approximations
of the exact value $u(0) = 5/64$. A $P1$ implementation of the operator
splitting approach leads to the expected approximation of the 
incorrect value $u_\infty(0)=3/64$. The Argyris element with simple
support boundary along the entire boundary (Argyris, full support) 
indicates convergence to another, lower and incorrect, value. 
A critical conditioning of the system matrix defined by the Argyris element 
leads to spurious values for fine meshes. 

\begin{figure}[htb]
\includegraphics[width=.75\linewidth]{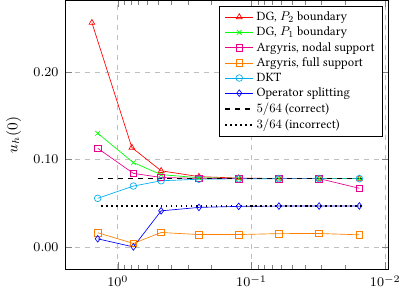} 
\caption{\label{fig:midpoint} Midpoint values for different numerical methods
and mesh sizes in Example~\ref{ex:bab_ex}. The Argyris method with simple support
condition on the entire boundary and the operator splitting approach lead to 
incorrect approximations.}
\end{figure}

The convergence behavior of the $H^2$ errors shown in Figure~~\ref{fig:h2-error}
confirms that not only the midpoint values are correctly approximated but convergence 
to the exact solution 
takes place for the Argyris method with nodal boundary condition,
the discrete Kirchhoff element, as well as for the affine and isoparametric
variants of the quadratic discontinuous Galerkin methods. While all methods
lead to positive convergence rates,
the nonconforming method leads to a quadratic experimental convergence rate despite
the involved domain approximations. 

\begin{figure}[htb]
\includegraphics[width=.75\linewidth]{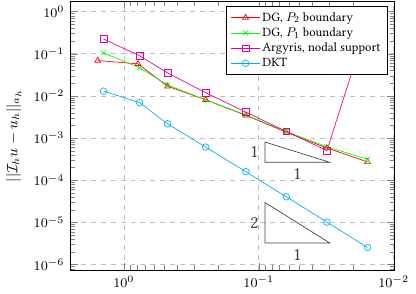}
\caption{\label{fig:h2-error} Experimental convergence behavior for the Argyris
method with boundary conditions imposed at the corner points, the discrete Kirchhoff
method, and quadratic discontinuous Galerkin methods.}
\end{figure}


\subsection*{Acknowledgments}
The authors thank Professor Andrea Bonito for pointing out the relevance
of the representation of the total curvature as a boundary integral. 
The authors are furthermore 
grateful to Professor Guido Sweers for stimulating discussions.
Support by the DFG via the priority programme SPP
2256 {\em Variational Methods for Predicting Complex Phenomena in Engineering Structures
and Materials} (BA 2268/7-2) is thankfully acknowledged.

\section*{References}
\printbibliography[heading=none]

@inproceedings {Davi02,
    AUTHOR = {Davini, Cesare},
     TITLE = {{$\Gamma$}-convergence of external approximations in boundary
              value problems involving the bi-{L}aplacian},
 BOOKTITLE = {Proceedings of the 9th {I}nternational {C}ongress on
              {C}omputational and {A}pplied {M}athematics ({L}euven, 2000)},
   JOURNAL = {J. Comput. Appl. Math.},
  FJOURNAL = {Journal of Computational and Applied Mathematics},
    VOLUME = {140},
      YEAR = {2002},
    NUMBER = {1-2},
     PAGES = {185--208},
      ISSN = {0377-0427},
   MRCLASS = {65N30 (49J45 65N12 74G15 76M25)},
  MRNUMBER = {1933237},
MRREVIEWER = {Enrique Fern\'{a}ndez Cara},
       DOI = {10.1016/S0377-0427(01)00525-8},
       URL = {https://doi.org/10.1016/S0377-0427(01)00525-8},
}

@incollection {Davi03,
    AUTHOR = {Davini, Cesare},
     TITLE = {Gaussian curvature and {B}abu\v{s}ka's paradox in the theory of
              plates},
 BOOKTITLE = {Rational continua, classical and new},
     PAGES = {67--87},
 PUBLISHER = {Springer Italia, Milan},
      YEAR = {2003},
   MRCLASS = {74K20 (35J35 46E35 74G65 74S05)},
  MRNUMBER = {2014180},
MRREVIEWER = {Juhani Pitk\"{a}ranta},
}

@article {CoNiSw19,
    AUTHOR = {De Coster, Colette and Nicaise, Serge and Sweers, Guido},
     TITLE = {Comparing variational methods for the hinged {K}irchhoff plate
              with corners},
   JOURNAL = {Math. Nachr.},
  FJOURNAL = {Mathematische Nachrichten},
    VOLUME = {292},
      YEAR = {2019},
    NUMBER = {12},
     PAGES = {2574--2601},
      ISSN = {0025-584X},
   MRCLASS = {35J40 (35J35 74K20)},
  MRNUMBER = {4056438},
MRREVIEWER = {Mario M. Coclite},
       DOI = {10.1002/mana.201800092},
       URL = {https://doi.org/10.1002/mana.201800092},
}

@article {BabPit90,
    AUTHOR = {Babu\v{s}ka, I. and Pitk\"{a}ranta, J.},
     TITLE = {The plate paradox for hard and soft simple support},
   JOURNAL = {SIAM J. Math. Anal.},
  FJOURNAL = {SIAM Journal on Mathematical Analysis},
    VOLUME = {21},
      YEAR = {1990},
    NUMBER = {3},
     PAGES = {551--576},
      ISSN = {0036-1410},
   MRCLASS = {73K10 (35J99 73C99)},
  MRNUMBER = {1046789},
MRREVIEWER = {Cornelius O. Horgan},
       DOI = {10.1137/0521030},
       URL = {https://doi.org/10.1137/0521030},
}

@article {NaSwSt11,
    AUTHOR = {Nazarov, S. A. and Sweers, G. and Stilyanou, A.},
     TITLE = {On paradoxes in problems of the bending of polygonal plates
              with ``hinge-supported'' edges},
   JOURNAL = {Dokl. Akad. Nauk},
  FJOURNAL = {Rossi\u{\i}skaya Akademiya Nauk. Doklady Akademii Nauk},
    VOLUME = {439},
      YEAR = {2011},
    NUMBER = {4},
     PAGES = {476--480},
      ISSN = {0869-5652},
   MRCLASS = {35J86 (35J25 35Q74 74G60 74K20)},
  MRNUMBER = {2893571},
}

@masterthesis{Wiss23-msc,
    AUTHOR = {Wissel, Marret},
    TITLE  = {Das {B}abu\v{s}ka-{P}aradoxon: {N}umerische {E}xperimente zu einer st\"uckweise linearen {R}andapproximation bei {B}iegeproblemen},
    year   = {2023},
    type   = {Master's thesis, University of Frei\-burg}
}

@article {BoNoNt21,
    AUTHOR = {Bonito, Andrea and Nochetto, Ricardo H. and Ntogkas,
              Dimitrios},
     TITLE = {D{G} approach to large bending plate deformations with
              isometry constraint},
   JOURNAL = {Math. Models Methods Appl. Sci.},
  FJOURNAL = {Mathematical Models and Methods in Applied Sciences},
    VOLUME = {31},
      YEAR = {2021},
    NUMBER = {1},
     PAGES = {133--175},
      ISSN = {0218-2025},
   MRCLASS = {65N30 (35Q74 65N12 74K20)},
  MRNUMBER = {4216049},
       DOI = {10.1142/S0218202521500044},
       URL = {https://doi.org/10.1142/S0218202521500044},
}

@article {UtkCar83,
    AUTHOR = {Utku, M. and Carey, G. F.},
     TITLE = {Penalty resolution of the {B}abu\v{s}ka circle paradox},
   JOURNAL = {Comput. Methods Appl. Mech. Engrg.},
  FJOURNAL = {Computer Methods in Applied Mechanics and Engineering},
    VOLUME = {41},
      YEAR = {1983},
    NUMBER = {1},
     PAGES = {11--28},
      ISSN = {0045-7825},
   MRCLASS = {65N30 (73K25)},
  MRNUMBER = {723043},
MRREVIEWER = {Igor Bock},
       DOI = {10.1016/0045-7825(83)90050-6},
       URL = {https://doi.org/10.1016/0045-7825(83)90050-6},
}

@article {ArnWal20,
    AUTHOR = {Arnold, Douglas N. and Walker, Shawn W.},
     TITLE = {The {H}ellan-{H}errmann-{J}ohnson method with curved elements},
   JOURNAL = {SIAM J. Numer. Anal.},
  FJOURNAL = {SIAM Journal on Numerical Analysis},
    VOLUME = {58},
      YEAR = {2020},
    NUMBER = {5},
     PAGES = {2829--2855},
      ISSN = {0036-1429},
   MRCLASS = {65N30 (35J40 74K20)},
  MRNUMBER = {4161311},
MRREVIEWER = {Gerrit Welper},
       DOI = {10.1137/19M1288723},
       URL = {https://doi.org/10.1137/19M1288723},
}

@inproceedings {Babu63,
    AUTHOR = {Babu\v{s}ka, I.},
     TITLE = {The theory of small changes in the domain of existence in the
              theory of partial differential equations and its applications},
 BOOKTITLE = {Differential {E}quations and {T}heir {A}pplications ({P}roc.
              {C}onf., {P}rague, 1962)},
     PAGES = {13--26},
 PUBLISHER = {Publ. House Czech. Acad. Sci., Prague},
      YEAR = {1963},
   MRCLASS = {35.00 (65.65)},
  MRNUMBER = {170133},
MRREVIEWER = {R. Carroll},
}

@article {BrNeSu13,
    AUTHOR = {Brenner, Susanne C. and Neilan, Michael and Sung, Li-Yeng},
     TITLE = {Isoparametric {$C^0$} interior penalty methods for plate
              bending problems on smooth domains},
   JOURNAL = {Calcolo},
  FJOURNAL = {Calcolo. A Quarterly on Numerical Analysis and Theory of
              Computation},
    VOLUME = {50},
      YEAR = {2013},
    NUMBER = {1},
     PAGES = {35--67},
      ISSN = {0008-0624},
   MRCLASS = {65N30 (74K20 74S05)},
  MRNUMBER = {3019146},
MRREVIEWER = {Minvydas Ragulskis},
       DOI = {10.1007/s10092-012-0057-1},
       URL = {https://doi.org/10.1007/s10092-012-0057-1},
}

@article {Rann79,
    AUTHOR = {Rannacher, R.},
     TITLE = {Finite element approximation of simply supported plates and
              the {B}abu\v{s}ka paradox},
   JOURNAL = {Z. Angew. Math. Mech.},
  FJOURNAL = {Zeitschrift f\"{u}r Angewandte Mathematik und Mechanik.
              Ingenieurwissenschaftliche Forschungsarbeiten},
    VOLUME = {59},
      YEAR = {1979},
    NUMBER = {3},
     PAGES = {T73--T76},
      ISSN = {0044-2267},
   MRCLASS = {65N30},
  MRNUMBER = {533989},
}

@book {StrFix73,
    AUTHOR = {Strang, Gilbert and Fix, George J.},
     TITLE = {An analysis of the finite element method},
    SERIES = {Prentice-Hall Series in Automatic Computation},
 PUBLISHER = {Prentice-Hall, Inc., Englewood Cliffs, NJ},
      YEAR = {1973},
     PAGES = {xiv+306},
   MRCLASS = {65N30},
  MRNUMBER = {443377},
MRREVIEWER = {R. E. Barnhill},
}

@article {MazNaz86,
    AUTHOR = {Maz{'}ya, V. G. and Nazarov, S. A.},
     TITLE = {Paradoxes of the passage to the limit in solutions of boundary
              value problems for the approximation of smooth domains by
              polygons},
   JOURNAL = {Izv. Akad. Nauk SSSR Ser. Mat.},
  FJOURNAL = {Izvestiya Akademii Nauk SSSR. Seriya Matematicheskaya},
    VOLUME = {50},
      YEAR = {1986},
    NUMBER = {6},
     PAGES = {1156--1177, 1343},
      ISSN = {0373-2436},
   MRCLASS = {35B40 (35J25 73K10)},
  MRNUMBER = {883157},
MRREVIEWER = {Michael Lorenz},
}

@article {DavPit00,
    AUTHOR = {Davini, Cesare and Pitacco, Igino},
     TITLE = {An unconstrained mixed method for the biharmonic problem},
   JOURNAL = {SIAM J. Numer. Anal.},
  FJOURNAL = {SIAM Journal on Numerical Analysis},
    VOLUME = {38},
      YEAR = {2000},
    NUMBER = {3},
     PAGES = {820--836},
      ISSN = {0036-1429},
   MRCLASS = {65N12 (65N30)},
  MRNUMBER = {1781205},
MRREVIEWER = {Riccardo Fazio},
       DOI = {10.1137/S0036142999359773},
       URL = {https://doi.org/10.1137/S0036142999359773},
}

@book {AtBuMi06,
    AUTHOR = {Attouch, Hedy and Buttazzo, Giuseppe and Michaille, G\'{e}rard},
     TITLE = {Variational analysis in {S}obolev and {BV} spaces},
    SERIES = {MPS/SIAM Series on Optimization},
    VOLUME = {6},
 PUBLISHER = {Society for Industrial and Applied Mathematics (SIAM),
              Philadelphia, PA},
      YEAR = {2006},
     PAGES = {xii+634},
      ISBN = {0-89871-600-4},
   MRCLASS = {49-02 (46E35 49J45 49J53 49K20 74G65)},
  MRNUMBER = {2192832},
MRREVIEWER = {Pablo Pedregal},
}

@book {ErnGue04,
    AUTHOR = {Ern, Alexandre and Guermond, Jean-Luc},
     TITLE = {Theory and practice of finite elements},
    SERIES = {Applied Mathematical Sciences},
    VOLUME = {159},
 PUBLISHER = {Springer-Verlag, New York},
      YEAR = {2004},
     PAGES = {xiv+524},
      ISBN = {0-387-20574-8},
   MRCLASS = {65-02 (65M60 65N30 74S05 76M10 78M10)},
  MRNUMBER = {2050138},
MRREVIEWER = {R. S. Anderssen},
       DOI = {10.1007/978-1-4757-4355-5},
       URL = {https://doi.org/10.1007/978-1-4757-4355-5},
}

@article {BGNY23,
    AUTHOR = {Bonito, Andrea and Guignard, Diane and Nochetto, Ricardo H.
              and Yang, Shuo},
     TITLE = {Numerical analysis of the {LDG} method for large deformations
              of prestrained plates},
   JOURNAL = {IMA J. Numer. Anal.},
  FJOURNAL = {IMA Journal of Numerical Analysis},
    VOLUME = {43},
      YEAR = {2023},
    NUMBER = {2},
     PAGES = {627--662},
      ISSN = {0272-4979},
   MRCLASS = {65M60 (74K20)},
  MRNUMBER = {4568427},
       DOI = {10.1093/imanum/drab103},
       URL = {https://doi.org/10.1093/imanum/drab103},
}

@book {GaGrSw10-book,
    AUTHOR = {Gazzola, Filippo and Grunau, Hans-Christoph and Sweers, Guido},
     TITLE = {Polyharmonic boundary value problems},
    SERIES = {Lecture Notes in Mathematics},
    VOLUME = {1991},
 PUBLISHER = {Springer-Verlag, Berlin},
      YEAR = {2010},
     PAGES = {xviii+423},
      ISBN = {978-3-642-12244-6},
   MRCLASS = {35-02 (31B30 35A08 35J40 35J61 46E35 58E12)},
  MRNUMBER = {2667016},
MRREVIEWER = {Rodney Josu\'{e} Biezuner},
       DOI = {10.1007/978-3-642-12245-3},
       URL = {https://doi.org/10.1007/978-3-642-12245-3},
}

@book {Bart16-book,
    AUTHOR = {Bartels, S\"{o}ren},
     TITLE = {Numerical approximation of partial differential equations},
    SERIES = {Texts in Applied Mathematics},
    VOLUME = {64},
 PUBLISHER = {Springer, [Cham]},
      YEAR = {2016},
     PAGES = {xv+535},
      ISBN = {978-3-319-32353-4; 978-3-319-32354-1},
   MRCLASS = {65-01 (65M60 65N30)},
  MRNUMBER = {3496530},
       DOI = {10.1007/978-3-319-32354-1},
       URL = {https://doi.org/10.1007/978-3-319-32354-1},
}

@book {Bart15-book,
    AUTHOR = {Bartels, S\"{o}ren},
     TITLE = {Numerical methods for nonlinear partial differential
              equations},
    SERIES = {Springer Series in Computational Mathematics},
    VOLUME = {47},
 PUBLISHER = {Springer, Cham},
      YEAR = {2015},
     PAGES = {x+393},
      ISBN = {978-3-319-13796-4; 978-3-319-13797-1},
   MRCLASS = {65-01 (35A15 35A35 65Mxx 65Nxx)},
  MRNUMBER = {3309171},
MRREVIEWER = {Karsten Urban},
       DOI = {10.1007/978-3-319-13797-1},
       URL = {https://doi.org/10.1007/978-3-319-13797-1},
}

@incollection {Scot77,
    AUTHOR = {Scott, L. Ridgway},
     TITLE = {A survey of displacement methods for the plate bending
              problem},
 BOOKTITLE = {Formulations and computational algorithms in finite element
              analysis ({U}.{S}.-{G}ermany {S}ympos., {M}ass. {I}nst.
              {T}ech., {C}ambridge, {M}ass., 1976)},
     PAGES = {855--876},
 PUBLISHER = {The M.I.T. Press, Cambridge, MA},
      YEAR = {1977},
   MRCLASS = {73.65},
  MRNUMBER = {475202},
MRREVIEWER = {M. N. L. Narasimhan},
}

@book {Brae07,
    AUTHOR = {Braess, Dietrich},
     TITLE = {Finite elements},
   EDITION = {Third},
 PUBLISHER = {Cambridge University Press, Cambridge},
      YEAR = {2007},
     PAGES = {xviii+365},
      ISBN = {978-0-521-70518-9; 0-521-70518-5},
   MRCLASS = {65N30 (65-02 74S05)},
  MRNUMBER = {2322235},
       DOI = {10.1017/CBO9780511618635},
       URL = {https://doi.org/10.1017/CBO9780511618635},
}

@article {Rann79b,
    AUTHOR = {Rannacher, Rolf},
     TITLE = {On nonconforming and mixed finite element method for plate
              bending problems. {T}he linear case},
   JOURNAL = {RAIRO Anal. Num\'{e}r.},
  FJOURNAL = {RAIRO Analyse Num\'{e}rique},
    VOLUME = {13},
      YEAR = {1979},
    NUMBER = {4},
     PAGES = {369--387},
      ISSN = {0399-0516},
   MRCLASS = {65N30 (73K10)},
  MRNUMBER = {555385},
       DOI = {10.1051/m2an/1979130403691},
       URL = {https://doi.org/10.1051/m2an/1979130403691},
}

\end{document}